\documentclass[11pt]{article}
\usepackage[margin=25mm]{geometry}

\usepackage{tabularx}
\usepackage{indentfirst}
\usepackage{setspace}

\usepackage{url}
\usepackage{color}
\usepackage{mathrsfs}

\usepackage[pdftex]{graphicx}

\usepackage{tikz-cd}
\usepackage{amscd}
\usepackage{enumerate}
\usepackage{mathtools}
\usepackage{amsmath}
\usepackage{amssymb}
\usepackage{amsthm}
\usepackage{amsfonts}
\usepackage{usebib}
\usepackage{comment}
\usepackage{graphics}

\newcommand{\ilim}[1][]{\mathop{\varinjlim}\limits_{#1}}

\newcommand{\abs}[1]{\left|#1\right|}

\newcommand{\N}{\mathbb{N}}
\newcommand{\Z}{\mathbb{Z}}
\newcommand{\Q}{\mathbb{Q}}
\newcommand{\R}{\mathbb{R}}
\newcommand{\C}{\mathbb{C}}
\newcommand{\F}{\mathbb{F}}

\newcommand{\Hom}{\mathrm{Hom}}

\newcommand{\Ext}{\mathrm{Ext}}

\newcommand{\SL}{\mathrm{SL}}
\newcommand{\PSL}{\mathrm{PSL}}

\newcommand{\restr}[2]{{
\left.\kern-\nulldelimiterspace
#1
\vphantom{\big|}
\right|_{#2}
}}

\theoremstyle{plain}
\newtheorem{Theorem}{Theorem}[section]
\newtheorem{Proposition}[Theorem]{Proposition}
\newtheorem{Lemma}[Theorem]{Lemma}
\newtheorem{Corollary}[Theorem]{Corollary}

\theoremstyle{definition}
\newtheorem{Definition}[Theorem]{Definition}
\newtheorem{Example}[Theorem]{Example}
\newtheorem{Remark}[Theorem]{Remark}

\newcommand{\ra}{\rightarrow}

\begin{document}

\title{Dijkgraaf-Witten invariant in topological $K$-theory}
\author{
Koki Yanagida 
\thanks{
\text{Department of Mathematics, Tokyo Institute of Technology, 2-12-1 Ookayama, Meguro-ku, Tokyo 152-8551, JAPAN}
\text{\quad \quad e-mail:\texttt{yanagida.k.ab@m.titech.ac.jp}
}
}
} 

\sloppy

\date{\empty}

\maketitle

\begin{abstract}
Given a finite group $G$, we define a new invariant of odd-dimensional oriented closed manifolds and call it the KDW invariant.
This invariant is a Dijkgraaf--Witten invariant in terms of $K$-theory.
In this paper, we compute the invariant of the Brieskorn homology spheres with $G=\PSL_2(\F_p)$.
We should remark that, in this computational result, the fundamental groups of the Brieskorn homology spheres and $\PSL_2(\F_p)$ are not nilpotent.
\end{abstract}

\begin{center}
\normalsize
\baselineskip=11pt
{\bf Keywords} \\
Dijkgraaf--Witten invariant, $K$-theory, Representation theory \ \ \
\end{center}
\begin{center}
\normalsize
\baselineskip=11pt
{\bf Subject Code } \\
\ \ \ \ \ \ 55N15
\ \ \
\end{center}
\large 

\section{Introduction}
For a finite group $G$, the Dijkgraaf--Witten (DW) invariant $\mathrm{DW}_G$ is a topological invariant of oriented closed $3$-manifolds defined by Dijkgraaf and Witten \cite{DijWit}.
Roughly speaking, the DW invariant, $\mathrm{DW}_G (M)$, of an oriented closed $3$-manifold $M$ is the formal sum of pushforwards $\varphi_*([M]) \in H_3(BG; \Z)$, where $\varphi$ ranges over all the homomorphisms $\pi_1 (M) \to G$.
Here, $[M] \in H_3(M,\Z)$ is an orientation class of $M$ and $BG$ is a classifying space of $G$.
In addition, Dijkgraaf and Witten have constructed a toy model-like $(2+1)$-TQFT derived from the DW invariant.
In other words, providing computational examples of the DW invariant leads to a deeper understanding of the TQFT.
Nevertheless, in the previous studies, there are few examples of computing the DW invariant unless $G$ is abelian or nilpotent.

As a generalization, we can define orientation classes of any generalized homology as analogies of the orientation class $[M]$ of ordinary homology (see, e.g., \cite{Rud}).
In particular, for the $K$-homology $K_*$, it is known that a spin$^c$ structure of an odd-dimensional oriented closed manifold $M$ defines an orientation class $[M]_K$ of $M$ in $K_* (M)$ (see \cite{HigRoe}). 
Thus,
thanks to these orientation classes,
it is possible to extend the definition of TQFTs based on ordinary homology to generalized homologies
under certain assumptions.
For TQFTs based on generalized homologies, see, e.g., \cite{Gom}.

This paper further studies the DW invariant in terms of $K$-theory; this form of the DW invariant is called the KDW invariant and denoted by $\mathrm{KDW}_G (M) \in \Z[K_1(BG)]$ (Definition \ref{oo1}).
Even though $K_1 (BG)$ has been determined in \cite{AtiSin,Eke}, 
it is difficult to precisely present the pushforwards of the orientation class $[M]_K$ via homomorphisms $\pi_1(M) \to G$;
thus the computation of the KDW invariant is often difficult.
To solve this difficulty, from \cite{Kna}, we focus on an embedding $\psi_G: K_1(BG) \otimes \Z[1/2] \hookrightarrow R(G) \otimes \Z_{(2)}/\Z$, and attempt to compute $\psi_G (\mathrm{KDW}_G(M))$.
Here, $R(G)$ is the complex representation ring.
In this paper, we call $\psi_G (\mathrm{KDW}_G(M))$ the $\alpha$-KDW invariant (Definition \ref{oo31}).
Through $\psi_G$, 
the computing of the KDW invariants without $2$-torsion follows from the computation of induced representations between the representation rings (see Remark \ref{REMori} for details).

This paper computes the $\alpha$-KDW invariant in the following three cases:
\begin{enumerate}[(i)]
\item The lens space $M=L^1(k;\ell)$ with $G=\Z/k$ (Example \ref{oo122}),
\item The lens space $M=L^1(k;\ell)$ with $G=\PSL_2(\F_p)$ (Proposition \ref{lenswithPSL}),
\item The Brieskorn homology sphere $M=\Sigma$ with $G=\PSL_2(\F_p)$ (Theorems \ref{withoutp} and \ref{withp}).
\end{enumerate}
We should remark that, in (iii), the fundamental group $\pi_1(M)$ and $G$ are not nilpotent.
Theorems \ref{withoutp} and \ref{withp} imply that, if $G$ is more complicated than $\PSL_2(\F_p)$ and nilpotent groups, then the computation of the DW invariants will be more difficult.

We now briefly explain procedures to obtain the results (i)-(iii): First, (i) follows from the resulting computation \cite{Kna} of the images of pushforward $\varphi_*([L^1(k;\ell)]_K) \in K_1(BG)$ for $\varphi : \pi_1(L^1(k;\ell)) \to G$ under $\psi_G$.
Next, we deduce (ii) from the naturality of $\psi_G$; for any injective group homomorphism $f:H \to G$, there exists $f_! : R(H) \to R(G)$ satisfying $\psi_G \circ f_* = f_! \circ \psi_{H}$.
Finally, to accomplish (iii), we focus on a bordism between $\Sigma$ and lens spaces from \cite{JonWes}.
Considering this bordism,  we can verify that the pushforward of $[\Sigma]_K$ is equal to the sum of the pushforwards of orientation classes of lens spaces.
Thus, as an application of (ii), classifying the conjugacy classes of $\Hom(\pi_1(\Sigma), \PSL_2(\F_p))$ is adequate to compute the $\alpha$-KDW invariant in (iii).

This paper is organized as follows:
Section $2$ defines the KDW invariant and the $\alpha$-KDW invariant, and presents the resulting computation of the $\alpha$-KDW invariants of lens spaces with finite cyclic groups.
In Section $3$, we provide the $\alpha$-KDW invariants of lens spaces with $\PSL_2(\F_p)$.
Section $4$ gives the computational results of the $\alpha$-KDW invariants of some Brieskorn homology spheres with $\PSL_2(\F_p)$ in Theorems \ref{withoutp} and \ref{withp}.
Section $5$ provides the proofs of the theorems stated in Section $4$.

\noindent 
\textbf{Conventional terminology.}
Throughout this paper, we fix an odd prime $p \in \N$ and denote by $\F_p$ the finite field of order $p$. 
Groups $G$ and $H$ are of finite order, and every tensor product $\otimes$ is over $\Z$. 
By $M$, we mean a closed connected manifold of dimension $d$ with orientation.
Let $I_2$ denote the identity matrix 
$ \begin{pmatrix}
1 & 0 \\
0 & 1 
\end{pmatrix} $.

\section{Preliminaries on $K$-orientations and $K_*(BG)$}\label{subsecr}

In this section, we review $K$-(co)homology and $K$-orientations from \cite{BauDou1,BauDou2,HigRoe,BauHigNig},
and define the DW invariant in terms of topological $K$-theory. 

We begin by recalling $K$-(co)homology. 
In general, given a compact Hausdorff space $X$, we can define the $K$-homology $K_*(X)$ and the $K$-cohomology $K^*(X)$. 
For example, $K^0(X)$ is the Grothendieck group of complex vector bundles on $X$,
and $K^{-n}(X)$ is defined to be $K^0( \Sigma^n X)$, where $\Sigma^n X$ is the $n$-fold suspension of $X$. 
On the other hand, we omit the definition of the $K$-homology. However, as seen in \cite{HigRoe},
if $X$ has a spin$^c$ structure $\mathfrak{s}$, there is a Poincar\'{e} duality 
\begin{equation}\label{eq1} \mathrm{PD}_{\mathfrak{s} } : K^*(X) \cong K_{d-*}(X). \end{equation}
Recall the Bott periodicity $K_{*+2}(X) \cong K_*(X) $ and $ K^{*+2}(X) \cong K^*(X) $. 
In consideration of this periodicity, we may focus only on the cases of $*=1$ and $*=0.$

Next, given a finite group $G$, we review the $K$-(co)homology of $G$.
Suppose a filtration of $G$-manifolds $ N_1 \subset N_2 \subset \cdots $ such that 
the actions on $N_*$ are smoothly free and the direct limit is contractible. 
Then, the $K$-homology, $K_*(BG)$, is defined to be the direct limit $\lim K_* (G \backslash N_n) $.
The definition is independent of the choices of the filtrations.
As is known (see, e.g., \cite{AtiSeg,Eke}), there is an additive isomorphism 
\begin{equation}\label{EQ:KbyR}
K_*(BG)\cong
\left\{
\begin{array}{ll}
\ilim \Hom_\Z (R(G)/I_G^n, \Z),& \text{if $*=0$,}\\
\ilim \Ext_\Z^1 (R(G)/I_G^n, \Z),& \text{if $*=1$.}\\
\end{array}
\right.
\end{equation}
Here, $R(G)$ is the complex representation ring, and $I_G \subset R(G)$ is the augmentation ideal.
Moreover, in \cite[Theorem 5.5.6]{Eke}, considering the linear isomorphism of $\C$-vector spaces between $R(G) \otimes \C$ and $\mathrm{class}(G)$, which is spanned by the irreducible characters of  $G$, we have
\begin{equation}\label{EQ:K(BG)}
K_*(BG)\cong 
\left\{
\begin{array}{ll}
\Z,& \text{if $*=0$,}\\
\prod_{p \in \mathcal{P}(G)} (\Z / p^\infty)^{r_G(p)},& \text{if $*=1$.}\\
\end{array}
\right.
\end{equation}
Here, $\mathcal{P}(G)$ is the set of primes that divide the order $|G|$, $r_G(p)$ is the number of the conjugacy classes consisting of elements whose orders are powers of $p$, and $ \Z / p^\infty$ is the Pr\"{u}fer $p$-group. 
In conclusion, by \eqref{EQ:K(BG)}, we should focus on the odd cases of $ \dim M$. 
\subsection{$K$-orientation and properties in three-dimensional cases}\label{inv2}
We review from \cite{Rud} orientations in homology theory. 
Let $E_*$ be any ring spectrum with associated cohomology and homology theory. 
For any point $x \in M$, suppose an open neighborhood $U_x \subset M$ with homeomorphism $U_x \cong \R^d.$
Then, the excision axiom yields the map
\[
h_x : E_*(M ) \ra E_*(M, M \setminus \{ x\} ) \cong E_*(U_x, U_x \setminus \{ x\} ) = E_*(S^d ).
\]
Then, {\it an orientation of $M$ on $E_*$} is defined to be a class $ [M]_E\in E_*(M)$ such that 
$h_x ([M]_E ) \in E_*(S^d )$ is a generator of $ E_*(S^d )$ for all $x \in M$. 

In addition, let us further consider the case of $E_*=K_*$, i.e., $E$ has the $K$-theory spectrum $KU$, and $(d -1 )/2 \in \Z$. 
Let $\eta: K_* \ra H_*$ be the edge homomorphism in the Atiyah-Hirzebruch homology spectral sequence. 
For example, $ \eta : K_1 (S^d )\ra H_d (S^d ;\Z) $ is an isomorphism. 
An orientation class $ [M]_K$ is said to be {\it on $ [M]$} if 
$[M]$ is equal to $ \eta \circ h_x ([M]_E ) \in H_d (S^d ;\Z)$ for any $x \in M.$
Let $\mathrm{KOri}([M])$ be the set of orientation classes on $[M]$. 

With the choice of an orientation class $[M]_K$, let us define the Dijkgraaf-Witten invariant in terms of $K$-theory as follows: 

\begin{Definition}\label{oo1}
For $[M]_K \in \mathrm{KOri}([M])$, we define {\it the $K$-theoretic Dijkgraaf-Witten invariant of $[M]_K$ with $G$} to be the formal sum
\[ \mathrm{KDW}_G([M]_K) = \sum_{ f \in \Hom(\pi_1(M) ,G) } 1_{\Z} (Bf \circ \iota_{\pi_1(M)})_* ([M]_K) \in \Z[ K_{* + \dim M} (BG)], \] 
where $\iota_{\pi_1(M)}$ is a classifying map $M \to B\pi_1 (M)$, and $Bf : B\pi_1 (M) \to BG$ is the continuous map induced from $f:\pi_1(M) \to G$.
For simplicity, we often write the {\it KDW invariant} for the invariant. 
\end{Definition}

Next, we focus on the case of $\dim M=3$.
Let us discuss the correspondence between spin$^c$ structures $\mathrm{Spin}^c (M)$ and $K$-orientation classes $\mathrm{KOri}([M])$ of $3$-manifolds. 
We can show the following proposition by a Mayer-Vietoris argument and the Chern classes.

\begin{Proposition}[\cite{KamSch,Mat}]\label{oo34}
Suppose $\dim M=3$.
Then, there is an isomorphism $K_1(M) \cong H_1 (M;\Z)\oplus H_3(M;\Z)$, 
and the set $\mathrm{KOri}([M])$ is bijective to $H_1 (M;\Z) $.
\end{Proposition}

Let us state Proposition \ref{oo35} below.
As is known, any spin$^c$-structure $ \mathfrak{s}$ defines the first Chern class $c_1(\mathfrak{s} ) \in H^2( M;\Z)$;
the correspondence $\mathfrak{s} \mapsto c_1(\mathfrak{s} )$ gives rise to a bijection between $\mathrm{Spin}^c(M)$ and $H^2( M;\Z) $.
Let $K^0(M) \ra H^2( M;\Z) $ be the homomorphism obtained by considering the first Chern-class. 
Moreover, if a spin$^c$-structure $ \mathfrak{s}_0$ is fixed, 
the composite $c_1 \circ (\mathrm{PD}_{\mathfrak{s}_0 })^{-1}$ yields a bijection between $\mathrm{KOri}([M])$ and $H^2( M;\Z) $.
In summary, 
\begin{Proposition}\label{oo35}
Suppose $\dim M=3$. Then, there are 1-1 correspondences among $\mathrm{KOri}([M])$, $ \mathrm{Spin}^c(M)$ and $ H_1( M;\Z) \cong H^2( M;\Z) $.
\end{Proposition}

\subsection{A homomorphism $\psi_G$ of Knapp}\label{inv.revised}
In general, it can be hard to deal with $K$-homology quantitatively.
We will give a modification of the KDW invariant from the viewpoint of Knapp \cite{Kna}.
This subsection contains nothing new. 
For any $g \in G$ and subgroup $H \subset G$, the conjugation $H \to g H g^{-1}$ canonically induces an isomorphism $R(H) \cong R(gHg^{-1})$;
under the isomorphism, we often identify $R(H)$ with $R(gHg^{-1})$.

Let us review the Knapp $G$-signature $\alpha^{\mathrm{rp}}_\Z$ in \cite{Kna,AtiSin}.
For a free $G$-manifold $N$ with an orientation, we canonically have a classifying map $\iota_G: N/G \ra BG $ from the $G$-principal bundle $N \ra N/G.$
Since the oriented bordism group with coefficients in $\Q/\Z$ of $BG$ is torsion, there exist $n_N \in \N$ and a free $G$-manifold $Y$ such that $\sharp^{n_N}N = \partial Y$. Then, as in \cite[(1.6)]{Kna}, we can define a class function $\alpha^{\mathrm{rp}}_{\Z}(\cdot,N)$ on $G$ as
\[
\alpha^{\mathrm{rp}}_{\Z}(g,N) = 
\left\{
\begin{array}{ll}
\frac{1}{n_N} {\rm Sign}(g,Y) , & \text{ if }g\neq 1, \\[9pt]
\frac{1}{n_N} {\rm Sign}(Y), & \text{ otherwise. }
\end{array}
\right.
\]
Here, ${\rm Sign}(g,Y)$ is the character function of the representation $ {\rm Sign}(G,Y)$ arising from the intersection form and the $G$-action on $H^*(Y;\Z)$; see \cite{AtiSin,Kna}.
Then, $\alpha^{\mathrm{rp}}_\Z (g,N)$ is known to be independent of the choice of $Y$ modulo $\Z$.
We denote by $\alpha^{\mathrm{rp}}_\Z(G,N) \in R(G)\otimes \Q/\Z$ the associated representation with the character $\alpha^{\mathrm{rp}}_{\Z}(\cdot,N)$.

Next, we now review an injective homomorphism
\[
\psi_G: K_1(BG) \otimes \Z[1/2] \longrightarrow R(G) \otimes \Z_{(2)}/\Z,
\]
defined in \cite{Kna}.
Let $ \Omega_*( X; \Q/\Z)$ be the oriented bordism group for a CW-complex $X$ in the coefficients $\Q/\Z$ (see \cite[p.106]{Kna} for the definition), $M$ be the quotient manifold $N/G$ of a free $G$-manifold of $N$, and $\iota_G : M \to BG$ be a classifying map for the $G$-principal bundle $N \to M$.
Then, as is known \cite[Proposition 1.9]{Kna}, the correspondence $(M , \iota_G : M \ra BG) \mapsto \alpha^{\mathrm{rp}}_\Z (G,N) $ gives rise to 
a homomorphism $\Omega_{\mathrm{odd}}(BG ; \Q/\Z ) \otimes \Z[1/2] \ra R(G) \otimes \Z_{(2)}/\Z $.
Here, $\Z_{(2)}$ denotes the localization of $\Z$ at $2$, and the quotient group $\Z_{(2)}/\Z$ is naturally isomorphic to $\Q/\Z \otimes \Z[1/2]$.
In addition, the Thom homomorphism $\tau: \Omega_{\rm odd}(BG; \Q/\Z ) \ra K_1(BG)$ is defined by
\[
\tau ( ( X, f ) ) = f_* [X]_s,
\]
where $[X]_s \in K_1(X)$ is an element defined in \cite[p.108]{Kna}.
Note that $[X]_s$ is an orientation class up to a factor $2^n$ if $X$ is $K$-orientable, and $\mathrm{Coker} (\tau)$ is a $2$-primary torsion group. 
Then, 
as in \cite{Kna},
$\alpha_\Z^{\mathrm{rp}}$ factors over $\mathrm{Im} (\mu_s \otimes \mathrm{Id}_{\Z[1/2]}) = K_1 (BG) \otimes \Z[1/2]$ to a map $K_1(BG) \otimes \Z[1/2] \ra R(G) \otimes \Z_{(2)}/\Z$. 
We will denote by $\psi_G$ the map $K_1(BG) \otimes \Z[1/2] \ra R(G) \otimes \Z_{(2)}/\Z$.

\begin{Definition}\label{oo31}
Suppose that $M$ is of odd dimension.
We define the formal sum
\[ \alpha \textrm{-KDW}_G(M) \coloneqq \sum_{ f \in \Hom(\pi_1(M) ,G) } 1_{\Z} \psi_G \circ (Bf \circ \iota_{\pi_1(M)})_* ([M]_s)\in \Z[ R(G) \otimes \Z_{(2)}/\Z ], \] 
where $[M]_s \in K_1(M)$ is the element defined in \cite[p.108]{Kna}.
We call the {\it $\alpha$-KDW invariant with $G$}. 
\end{Definition}
\begin{Remark}\label{REMori}
Since $\psi_G$ is injective modulo $2$-torsion \cite[Theorem 2.10]{Kna}, for an orientation class $[M]_K \in K_1 (M)$ satisfying $[M]_K \otimes 1 = [M]_s \otimes 1 \in K_1(M) \otimes \Z[1/2]$, we can identify $\textrm{KDW}_G([M]_K)$ with $\textrm{$\alpha$-KDW}_G(M)$ without considering $2$-torsion.
\end{Remark}

Next, we note a naturality of the above map $\psi_G $. 
Let $\varphi:H \ra G$ be a group homomorphism between two finite groups. 
The correspondence from a representation of $H$ to the induced representation of $G$ gives rise to an 
additive homomorphism $\varphi_!:R(H) \ra R(G) $. 
Then, 
\begin{equation}\label{n2}
\psi_G \circ B\varphi_* = \varphi_! \circ \psi_H
\end{equation}
holds; see \cite[Theorem 2.1]{Kna}.
In particular, if $H$ is a subgroup of $G$ and $\varphi$ is the inclusion, 
then $\varphi_!$ is equal to the induced representation $\mathrm{Ind}^G_H : R(H) \to R(G)$, which is an additive homomorphism of $\Z$-modules.
Here, for a representation $V$ of $H$, $\mathrm{Ind}^G_H V$ is defined by the representation $\C[G] \otimes_{\C[H]} V$ of $G$.
Note that $\mathrm{Ind}^G_H = \mathrm{Ind}^G_{gHg^{-1}}$ under the identification $R(H) = R(gHg^{-1})$.
Thus, if a homomorphism $\tilde{f}: \pi_1(M) \to G$ is decomposed as
\[
\raisebox{-0.5\height}{\includegraphics[scale=1]{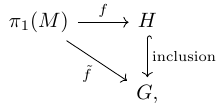}}
\]
then
\begin{equation}\label{EQ:Ind}
\psi_G \circ B\tilde{f}_* \circ \iota_{\pi_1(M)} = 
\mathrm{Ind}^G_H \circ (\psi_H \circ Bf_* \circ \iota_{\pi_1(M)})
\end{equation}
holds by the naturality (\ref{n2}).

Next, we consider the special case of $\pi_1(N)=1$, i.e., $\pi_1(M) = G$.
\begin{Proposition}\label{oo31}
Suppose that $N$ is a simply connected free $G$-manifold and $M=N/G$ is a $K$-oriented manifold of odd dimension.
Then,
\begin{equation}\label{DEF:alphaK}
\alpha \textrm{-}\mathrm{KDW}_G(M) = \sum_{ f \in \Hom(\pi_1(M) ,G) } 1_{\Z} (f_! (\alpha^{\rm rp}_\Z (G,N)) )\in \Z[ R(G) \otimes \Z_{(2)}/\Z ].
\end{equation}
\end{Proposition}
\begin{proof}
For any $f \in \Hom (\pi_1(M), G)$, the diagram
\[\raisebox{-0.5\height}{\includegraphics[scale=1]{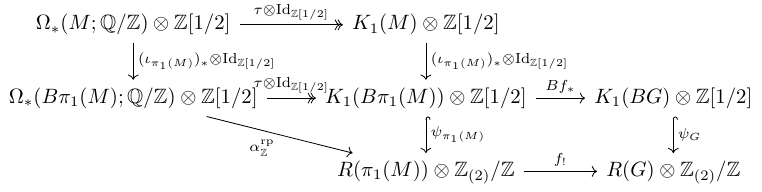}}\]
commutes.
Since $\tau([M,\mathrm{id}_M]) = [M]_s$ and $(\iota_{\pi_1 (M)})_*([M,\mathrm{id}_M]) = [M,\iota_{\pi_1 (M)}]$,
we have
\begin{equation}\label{EQ:psiG}
\psi_G \circ (Bf \circ \iota_{\pi_1(M)})_* ( [M]_s )= f_! \circ \alpha^{\mathrm{rp}}_\Z (G,N).
\end{equation}
We get (\ref{DEF:alphaK}) by applying (\ref{EQ:psiG}) to the definition of the $\alpha$-KDW invariant.
\end{proof}

\begin{Example}\label{oo122}
Let us compute the $\alpha$-KDW invariants of the lens spaces.

First, we will review the definition of the (generalized) lens space.
Let $G=\Z/k$ and $N=S^{2m+1}$ for some $k,m \in \N$.
Let $ \lambda \in \mathbb{C}$ be $e^{2 \pi \sqrt{-1} /k} $.
Fix positive integers $\ell_1, \dots, \ell_m$ such that $ k$ and each $\ell_i $ are relatively prime. 
Then, the cyclic group $\Z/k $ freely acts on the sphere 
\[
S^{2m+1} = \{ (z_0, \dots, z_m) \in \C^{m+1} \mid \abs{z_0}^2 + \abs{z_1}^2 + \cdots + \abs{z_m}^2 =1 \}
\]
by $
h \cdot (z_0, \dots, z_m) \coloneqq ( \lambda^{h} z_0, \lambda^{h\ell_1} z_n,\dots, \lambda^{h\ell_m} z_m)$, where $h \in \Z/k$. 
Then, the lens space $L^m (k; \ell_1 , \ldots , \ell_m)$ is defined to be the quotient space of $S^{2m+1}$ subject to $\Z/p$.
Hereafter, we denote by $M$ the lens space $L^m (k; \ell_1 , \ldots , \ell_m)$.
The loop $[0,1] \ra M$ defined by $t \mapsto ( \lambda^{ t },0,\dots, 0)$ gives a generator $\gamma$ of $\pi_1(M)$.
In this paper, we call this generator \emph{the canonical generator}.
We often identify $\pi_1 (M)$ with $\Z/k$ through the canonical generator.

Next, we will establish an additive isomorphism $\mathcal{T}$ between the representation ring $R(\Z/k)$ and $\Z^k$ as follows.
Let $\rho$ be the representation $ \Z/k \ra \mathrm{GL}_1 (\mathbb{C})$ that sends $1$ to $ \lambda$, which is irreducible.
Then, it is known (see, e.g., \cite{Seg}) that $R(\Z/k)$ is the same as the quotient ring $\Z[\rho]/(\rho^k - 1)$.
We define the isomorphism $\mathcal{T} : R(\Z/k) \to \Z^k$ by
\[
\mathcal{T}(\sum_{i=0}^{k-1} a_i \rho^{i}) = (a_0, \ldots , a_{k-1})^{\textrm{transpose}}.
\]

Finally, we give the $\alpha$-KDW invariants of $M$ with $\Z/k$.
For $h\in \Z/k \setminus \{ 0 \}$, let us define the homomorphism $f^{(h)} : \pi_1(M) \to \Z/k$ by $f^{(h)}(\gamma) = h$.
Note that $\Hom (\pi_1(M), \Z/k) = \{ e \} \sqcup \{ f^{(h)} \,\,|\,\, h \in \Z/k \setminus \{ 0 \} \}$, where $e$ means the trivial homomorphism.
In addition, for $h \in \Z/k \setminus \{ 0 \}$, let $P^{(h)}_k$ be the $(k \times k)$-permutation matrix whose $(i,j)$-entry is given by
\[
(\text{$P^{(h)}_k$'s $(i,j)$-entry}) =
\left\{
\begin{array}{ll}\displaystyle
1, & \text{ if $j \equiv h i \mod k$} , \\
0, & \text{ otherwise. }
\end{array}
\right.
\]
With the notation
\[
\boldsymbol{\xi}(k;\ell_1,\ldots,\ell_m) \coloneqq (\mathcal{T} \otimes \mathrm{Id}_{\Z_{(2)}/\Z}) (\alpha^{\mathrm{rp}}_{\Z}(\Z/k,S^{2m+1})),
\]
we have 
\begin{equation}\label{EQ:alphaLens}
(\mathcal{T} \otimes \mathrm{Id}_{\Z_{(2)}/\Z}) \circ f^{(h)}_! ( \alpha_\Z^{\mathrm{rp}} (\Z/k ; S^{2m+1} ))
= P^{(h)}_k \boldsymbol{\xi}(k;\ell_1,\ldots,\ell_m).
\end{equation}
Therefore, by Proposition \ref{oo31}, we have 
\[
\text{$\alpha$-KDW}_G (M) = 1_\Z \boldsymbol{o} + \sum_{h=1}^{k-1} 1_{\Z} (\mathcal{T} \otimes \mathrm{Id}_{\Z_{(2)}/\Z})^{-1} (P^{(h)}_k \boldsymbol{\xi}(k;\ell_1,\ldots,\ell_m) ).
\]
Here, $\boldsymbol{o} \in R(\Z/k) \otimes \Z_{(2)}/\Z$ denotes the additive zero.

Hereafter, we fix an isomorphism $\mathcal{T} \otimes \mathrm{Id}_{\Z_{(2)}/\Z}$, and identify $R(\Z/k) \otimes \Z_{(2)}/\Z$ with $(\Z_{(2)}/\Z)^k$.
Note that we consider $\Z^k \otimes \Z_{(2)}/\Z$ as $(\Z_{(2)}/\Z)^k$ under the correspondence 
\[
(a_0, \ldots , a_{k-1})^{\textrm{transpose}} \otimes b \longleftrightarrow (a_0 b, \ldots , a_{k-1} b)^{\textrm{transpose}}.
\]
In particular, from (\ref{EQ:alphaLens}), we also identify $P^{(h)}_k$ and $\boldsymbol{\xi}(k;\ell_1,\ldots,\ell_m)$ with $f^{(h)}_!$ and $\alpha_\Z^{\mathrm{rp}} (\Z/k ; S^{2m+1} )$, respectively.

As is known \cite[p.108]{Kna},
\[\displaystyle
\alpha^{\mathrm{rp}}_{\Z}(h,S^{2m+1}) = 
\left\{
\begin{array}{ll}\displaystyle
-\frac{\lambda^{h} +1}{\lambda^{h} -1} \prod_{i=1}^{m} \frac{\lambda^{h \ell_i} +1}{\lambda^{h \ell_i} -1}, & \text{ if }h\in \Z/k \setminus \{ 0 \}, \\
0, & \text{ otherwise. }
\end{array}
\right. 
\]
Since $\alpha^{\mathrm{rp}}_{\Z}(\cdot,S^{2m+1})$ is the character of the representation $\alpha^{\mathrm{rp}}_{\Z}(\Z/k,S^{2m+1})$, we can concretely compute $\alpha^{\mathrm{rp}}_{\Z}(\Z/k,S^{2m+1}) \in R(\Z/k) \otimes \Z_{(2)}/\Z$, in particular, and we can get $\boldsymbol{\xi}(k;\ell_1,\ldots,\ell_m) \in (\Z_{(2)}/\Z)^k$.

Let us consider the case of $m=1$.
Here, for $k\in\{ 3,5,7,11,13 \}$ and $1 \leq \ell \leq (k-1)/2$, we give a list of the resulting computation of $\boldsymbol{\xi}(k;\ell)$.
Since $\boldsymbol{\xi}(k;\ell) = -\boldsymbol{\xi}(k;-\ell')$ for any $\ell$ and $\ell'$ such that $\ell \equiv \ell' \mod k$, we can compute $\boldsymbol{\xi}(k;\ell)$ for any $k\in\{ 3,5,7,11,13 \}$ and $\ell\in\Z \setminus k\Z$ from Table \ref{TAB:xi}.
\begin{table}[bpth]
\center
\begin{tabular}{ll}
$(k;\ell)$& $\boldsymbol{\xi}(k;\ell)$ \\ \hline \hline
$(3;1)$ & $(2,8,8)/9$ \\ \hline
$(5;1)$ & $(4,1,2,2,1)/5$ \\ 
$(5;2)$ & $(0,1,4,4,1)/5$ \\ \hline
$(7;1)$ & $(3,5,4,0,0,4,5)/7$ \\ 
$(7;2)$ & $(2,3,6,4,4,6,3)/7$ \\ 
$(7;3)$ & $(5,1,3,4,4,3,1)/7$ \\ \hline
$(11;1)$ & $(8,10,5,4,7,3,3,7,4,5,10)/11$ \\ 
$(11;2)$ & $(10,0,3,8,4,2,2,4,8,3,0)/11$ \\ 
$(11;3)$ & $(6,3,5,1,2,8,8,2,1,5,3)/11$ \\ 
$(11;4)$ & $(6,1,8,5,3,2,2,3,5,8,1)/11$ \\
$(11;5)$ & $(1,8,7,9,3,0,0,3,9,7,8)/11$ \\\hline
$(13;1)$ & $(5,7,0,10,11,3,12,12,3,11,10,0,7)/13$ \\ 
$(13;2)$ & $(3,4,7,12,6,2,0,0,2,6,12,7,4)/13$ \\ 
$(13;3)$ & $(4,9,11,10,6,12,2,2,12,6,10,11,9)/13$ \\ 
$(13;4)$ & $(9,3,11,7,4,2,1,1,2,4,7,11,3)/13$ \\ 
$(13;5)$ & $(0,3,12,1,9,10,4,4,10,9,1,12,3)/13$ \\
$(13;6)$ & $(10,6,7,0,11,1,9,9,1,11,0,7,6)/13$ \\ 
\end{tabular}
\caption{$\boldsymbol{\xi}(k;\ell)$ for $k\in\{3,5,7,11,13\}$ and $1 \leq \ell \leq (k-1)/2$.\label{TAB:xi}}
\end{table}
\end{Example}

\section{$\alpha$-KDW invariant of lens spaces with  $\PSL_2(\F_p)$}\label{sec3}
In this section, we will compute the $\alpha$-KDW invariants of lens spaces $L(k;\ell)$ with $\PSL_2(\F_p)$. 
Since the fundamental groups of lens spaces  are cyclic groups, the images of non-trivial homomorphisms $\pi_1(L(k;\ell)) \to \PSL_2(\F_p)$ are cyclic subgroups.
Hence, based on (\ref{EQ:Ind}), 
it is sufficient to study the representations induced from cyclic subgroups of $\PSL_2(\F_p)$, 
which correspond to images of non-trivial homomorphisms from the fundamental group of a lens space.
Throughout this subsection, we mean by $G$ the projective special linear group $\PSL_2(\F_p)$.
Note that the order of $G$ is $(p^3-p)/2$.

\subsection{Group structure of $\PSL_2(\F_p)$}\label{pslk}
We first overview the group structure of $G$.
Let $\F_{p^2}$ be the finite field of order $p^2$.
Let $\mu_{p-1} = \F_p^\times$, and $\mu_{p+1}\subset \F_{p^2}$ be the cyclic subgroup consisting of $(p+1)$-th roots of unity in $\F_{p^2}$.
Fix a non-square number $\Delta \in \F_p$.
Then, with a choice of $\sqrt{\Delta} \in \F_{p^2}$, it follows that $\F_{p^2} = \F_p \oplus \F_p \sqrt{\Delta}$ as $\F_p$-modules, and $\mu_{p+1} = \left\{ x+y\sqrt{\Delta} \mid x^2 - \Delta y^2 = 1 \right\}$.
In this paper, we hereafter fix the arbitrary generators  $\zeta_- \in \mu_{p-1}$ and $ \zeta_+ \in \mu_{p+1}$.
Let us define respective homomorphisms $\sigma_{T},\sigma_{U}$, and $\sigma_{B}$ by
\begin{align*}
&\sigma_{T} : \mu_{p-1} \longrightarrow \PSL_2(\F_p);&&x \longmapsto
\begin{pmatrix}
x & 0 \\
0 & x^{-1} 
\end{pmatrix},\\[4pt]
&\sigma_{U} : \F_p \longrightarrow \PSL_2(\F_p);&&x \longmapsto
\begin{pmatrix}
1 & x \\
0 & 1
\end{pmatrix},\\[4pt]
&\sigma_{B} : \mu_{p+1} \longrightarrow \PSL_2(\F_p);&&x+\sqrt{\Delta} y \longmapsto
\begin{pmatrix}
x & \Delta y \\
y & x 
\end{pmatrix}.
\end{align*}
Furthermore, take three subgroups of the forms
\[
T\coloneqq\mathrm{Im} \sigma_{T},
\quad U\coloneqq\mathrm{Im}\sigma_{U},
\quad B\coloneqq\mathrm{Im} \sigma_{B}.
\]
The restriction maps $\sigma_{T} |_{\mu_{p-1} \setminus \{ 1 \}} $ and $\sigma_{B} |_{\mu_{p+1} \setminus \{ 1 \}} $ are $2$-$1$ maps satisfying $\sigma_{T}(\zeta_-^{i}) = \sigma_{T}(\zeta_-^{\frac{p-1}{2} - i})$ and $\sigma_{B}(\zeta_+^{i}) = \sigma_{B}(\zeta_+^{\frac{p+1}{2} - i})$, respectively.
Meanwhile, $\sigma_{U}$ is a bijection.
Therefore, we have isomorphisms
\[
T \cong \Z/((p-1)/2), \qquad U \cong \Z/p, \qquad B \cong \Z/((p+1)/2).
\]
As a result of \cite{Dor}, a complete set of class representatives for the conjugacy action in $G$ has been classified in Tables \ref{TAB:rep1} and \ref{TAB:rep2}.
Here, for $x \in G$, we denote by $\mathrm{Cl}_G(x)$ the conjugacy class of $x$ in $G$.\newline
\begin{table}[htpb]
\centering
\begin{tabular}{l||l|l|l|l|l|l}
Class rep. $G$ & $I_2$ & $\sigma_{U}(1)$ & $\sigma_{U}(\Delta)$ & $\sigma_{T}(\zeta_-^\ell)$ & $\sigma_{T}(\zeta_-^{\frac{p-1}{4}})$ & $\sigma_{B}(\zeta_+^m)$ \\\hline 
$|\mathrm{Cl}_G(x)|$ & $1$ & $\dfrac{p^2-1}{2}$ & $\dfrac{p^2-1}{2}$ & $p^2+p$ & $\dfrac{p^2+p}{2}$ & $p^2-p$ \rule[0mm]{0mm}{7mm} \\[8pt] \hline
Order of $x$ & $1$ & $p$ & $p$ & $\dfrac{p-1}{2((p-1)/2,\ell)}$ & $2$ & $\dfrac{p+1}{2((p+1)/2,m)}$ \rule[0mm]{0mm}{7mm} 
\end{tabular}
\caption{Class representatives of the conjugacy classes when $p \equiv 1 \text{ modulo $4$}$, for $1 \leq \ell \leq (p-1)/4 -1$ and $1 \leq m \leq (p-1)/4$.\label{TAB:rep1}}
\end{table}
\begin{table}[htpb]
\centering
\begin{tabular}{l||l|l|l|l|l|l}
Class rep. $G$ & $I_2$ & $\sigma_{U}(1)$ & $\sigma_{U}(\Delta)$ & $\sigma_{T}(\zeta_-^\ell)$ & $\sigma_{B}(\zeta_+^m)$ & $\sigma_{B}(\zeta_+^{\frac{p+1}{4}}) $ \\ \hline
$|\mathrm{Cl}_G(x)|$ & $1$ & $\dfrac{p^2-1}{2}$ & $\dfrac{p^2-1}{2}$ & $p^2+p$ & $p^2-p$ & $\dfrac{p^2-p}{2}$ \rule[0mm]{0mm}{7mm} \\[8pt] \hline
Order of $x$ & $1$ & $p$ & $p$ & $\dfrac{p-1}{2((p-1)/2,\ell)}$ & $\dfrac{p+1}{2((p+1)/2,m)}$ & $2$ \rule[0mm]{0mm}{7mm}
\end{tabular}
\caption{Class representatives of the conjugacy classes when $p \equiv -1 \text{ modulo $4$}$, for $1 \leq \ell \leq (p-3)/4$ and $1 \leq m \leq (p+1)/4-1$.\label{TAB:rep2}}
\end{table}

For $n \in \N$ and a prime $k$, let us define $\nu_k(n)$ by $\nu_k(n) \coloneqq \mathrm{max}\{ m \in \N \cup \{ 0\} \mid k^m \text{ divides } n \}$.
A careful verification with Tables \ref{TAB:rep1} and \ref{TAB:rep2} leads us to the following corollary:
\begin{Corollary}\label{BPSL}
Let $G=\PSL_2(\F_p)$ and $k \in \mathcal{P}(G)$. Then,
\[
r_{G} (k)=
\left\{
\begin{array}{ll}
\displaystyle
\frac{k^{\nu_k ((p-1)/2)} - 1}{2}, & \text{ if $k \notin \{ 2,p\}$ and $k | (p-1)$}, \\[10pt]
\displaystyle
\frac{k^{\nu_k ((p+1)/2)} - 1}{2}, & \text{ if $k \notin \{ 2,p\}$ and $k | (p+1)$}, \\[10pt]
\displaystyle
2^{\nu_2 ((p-1)/2)-1}, & \text{ if $k=2$ and $p \equiv 1\mod 4$}, \\[10pt]
\displaystyle
2^{\nu_2 ((p+1)/2)-1}, & \text{ if $k=2$ and $p \equiv -1\mod 4$}, \\[10pt]
\displaystyle
2, & \text{ if $k=p$}.
\end{array}
\right.
\]
\end{Corollary}
From (\ref{EQ:K(BG)}), this corollary determines the structure of $K_1(BG)$. 

\subsection{Induced representations of $\PSL_2(\F_p)$ from some cyclic subgroups}\label{pslind}
Let $k$ be an odd prime factor of $|G|$, and $\ell$ be an integer coprime to $k$.
In this subsection, for $n\in\N$ and $i \in \Z/n$, we define a representation $\rho_i^{(n)}$ of $\Z/n$ by $\rho_i^{(n)} (j) = e^{2 \pi \sqrt{-1} ji/n}$.
In addition, $\chi^{(n)}_i$ denotes the character corresponding to $\rho^{(n)}_i$.
Considering $\sigma_{T}(\zeta_-)$ (resp. $\sigma_{U}(1), \sigma_{B}(\zeta_+)$) as a generator of $T$ (resp. $U, B$), we use the same letter $\rho_i^{(\frac{p-1}{2})}$ (resp. $\rho_i^{(p)}, \rho_i^{(\frac{p+1}{2})}$).
We note that, for any finite group $L$, the representation ring $R(L)$ is a free $\Z$-module spanned by the irreducible representations of $L$.

\begin{Lemma}\label{LEM:IndT}

\noindent
\rm{(i)}
If $p \equiv 1 \mod 4$, then
$\mathrm{Im}( \mathrm{Ind}^G_T ) \subset R(G)$ is a free submodule of rank $(p+3)/4$ generated by
\begin{equation}\label{genIndT}
\left\{
\mathrm{Ind}^G_T \rho^{(\frac{p-1}{2})}_i
\middle|
0 \leq i \leq (p-1)/4
\right\}.
\end{equation}

\noindent
\rm{(ii)}
If $p \equiv -1 \mod 4$, then 
$\mathrm{Im}( \mathrm{Ind}^G_T ) \subset R(G)$ is a free submodule of rank $(p+1)/4$ generated by
\begin{equation}\label{genIndTT}
\left\{
\mathrm{Ind}^G_T \rho^{(\frac{p-1}{2})}_i
\middle|
0 \leq i \leq (p-3)/4
\right\}.
\end{equation}

\noindent
Moreover, 
in both cases (i) and (ii), the following equality holds for any $i \in \Z/( (p-1)/2 )$:
\begin{equation}\label{EQ:LEM:IndT}
\mathrm{Ind}^G_T \rho^{(\frac{p-1}{2})}_i = \mathrm{Ind}^G_T \rho^{(\frac{p-1}{2})}_{ -i}.
\end{equation}
\end{Lemma}
\begin{proof}
As is known, the correspondence from a representation to its character induces an isomorphism of $\C$-vector spaces $R(G) \otimes \C \to \mathrm{class}(G)$.
Therefore, it is sufficient to examine the induced character $\mathrm{Ind}^G_T \chi^{(\frac{p-1}{2})}_i$ in $\mathrm{class}(G)$, instead of $\mathrm{Ind}^G_T \rho^{(\frac{p-1}{2})}_i$ in $R(G)$.

(i) Suppose that $p \equiv 1 \mod 4$.
The conjugacy classes of $T$ under $\PSL_2(\F_p)$ are
\[
\{ I_2 \}, \,\,
\{\sigma_{T}(\zeta_-^{\frac{p-1}{4}}) \},\,\,
\{\sigma_{T}(\zeta_-^{\frac{p-1}{4} - 1}), \sigma_{T}(\zeta_-^{\frac{p-1}{4} + 1})\},\,\,
\{\sigma_{T}(\zeta_-^{\frac{p-1}{4} - 2}), \sigma_{T}(\zeta_-^{\frac{p-1}{4} + 2})\},\,\,
\ldots,
\{\sigma_{T}(\zeta_-), \sigma_{T}(\zeta_-^{\frac{p-3}{2}})\}.
\]
Hence, (\ref{EQ:LEM:IndT}) is obtained from
\begin{equation*}
(\mathrm{Ind}^G_T \chi^{(\frac{p-1}{2})}_i)(x)=
\left\{
\begin{array}{ll}
p(p+1), &\text{ if $x$ is conjugate to $I_2$,}\\
(-1)^i 2 \mathrm{cos}( 4 \pi \, \dfrac{ij}{p-1} ), &\text{ if $x$ is conjugate to $\sigma_T (\zeta_{-}^{\frac{p-1}{4} - j})$ for $j \in \{0, 1, \ldots, \frac{p-5}{4}\}$,}\\
0, &\text{ otherwise. }
\end{array}
\right.
\end{equation*}

Next, let us show that (\ref{genIndT}) is linearly independent over $\C$. Suppose that $a_1, a_2, \ldots, a_{(p-1)/4} \in \C$ satisfy
\begin{equation}\label{EQ:genIndT}
\sum_{i=0}^{(p-1)/4} (-1)^i a_i \mathrm{Ind}^G_T \chi^{(\frac{p-1}{2})}_i (x)= 0
\end{equation}
for any $x \in G$.
For $j \in \{ 0, 1,\ldots , (p-1)/2 \} \setminus\{ (p-1)/4 \}$, taking $ x = \sigma_{T}(\zeta_-^{\frac{p-1}{4} - j})$ in (\ref{EQ:genIndT}), we obtain
\begin{equation}\label{EQ:exponent}
\sum_{i=0}^{(p-1)/4} a_i e^{2 \pi \sqrt{-1} \frac{2ji}{p-1} } +
\sum_{i=(p-1)/4}^{(p-1)/2} a_{\frac{p-1}{2}-i} e^{2 \pi \sqrt{-1} \frac{2ji}{p-1} }
=0.
\end{equation}
In addition, taking $x = I_2$ in (\ref{EQ:genIndT}), we can also see that (\ref{EQ:exponent}) holds for $j=(p-1)/4$.
Now, let us denote by $\boldsymbol{w}_i$ a $( (p-1)/2 )$-dimensional vector whose $i$-th component is $e^{2 \pi \sqrt{-1} \frac{2ji}{p-1}}$.
Then, (\ref{EQ:exponent}) gives rise to
\[
2 a_0 \boldsymbol{w}_0 
+ \sum_{i=1}^{(p-5)/4} a_i \boldsymbol{w}_i 
+ 2 a _{\frac{p-1}{4}} \boldsymbol{w}_{\frac{p-1}{4}} 
+ \sum_{i=(p+3)/4}^{(p-3)/2} a_{\frac{p-1}{2}-i} \boldsymbol{w}_i = \boldsymbol{o}.
\]
Since $\{ \boldsymbol{w}_0, \boldsymbol{w}_1, \ldots , \boldsymbol{w}_{(p-3)/2}\}$ is linearly independent over $\C$, 
all $a_i$ are zero.
Therefore, (\ref{genIndT}) is linearly independent.

(ii) Suppose $p \equiv -1 \mod 4$.
The conjugacy classes of $T$ under $\PSL_2(\F_p)$ are
\[
\{ I_2 \}, \,\,
\{\sigma_{T}(\zeta_-^{\frac{p-1}{4} - \frac{1}{2}}), \sigma_{T}(\zeta_-^{\frac{p-1}{4} + \frac{1}{2}})\}, \,\,
\{\sigma_{T}(\zeta_-^{\frac{p-1}{4} - \frac{3}{2}}), \sigma_{T}(\zeta_-^{\frac{p-1}{4} + \frac{3}{2}})\},
\ldots,
\{\sigma_{T}(\zeta_-), \sigma_{T}(\zeta_-^{\frac{p-3}{2}})\}.
\]
Therefore, (\ref{EQ:LEM:IndT}) follows from
\begin{equation*}
(\mathrm{Ind}^G_T \chi^{(\frac{p-1}{2})}_i)(x)=
\left\{
\begin{array}{ll}
p(p+1), &\text{ if $x=I_2$,}\\
(-1)^i 2 \mathrm{cos}( 2 \pi \, \dfrac{(2j-1)i}{p-1} ), &\text{ if $x=\sigma_{T}(v^{\frac{p-1}{4} - \frac{2j-1}{2}})$ and $j \in \{1, 2, \ldots, \frac{p-3}{4}\}$,}\\
0, &\text{ otherwise. }
\end{array}
\right.
\end{equation*}
Moreover, the linearly independence of (\ref{genIndTT}) follows in a similar fashion to (i).
\end{proof}

Regarding $B$ and $U$, similar lemmas can be stated as follows.
\begin{Lemma}\label{LEM:IndB}
\rm{(i)}
If $p \equiv 1 \mod 4$, then
$\mathrm{Im}( \mathrm{Ind}^G_{B} ) \subset R(G)$ is a free submodule of rank $(p+3)/4$ generated by
\[
\left\{
\mathrm{Ind}^G_{B} \rho^{(\frac{p+1}{2})}_i
\middle|
0 \leq i \leq (p-1)/4
\right\}.
\]

\noindent
\rm{(ii)}If $p \equiv -1 \mod 4$, then
$\mathrm{Im}( \mathrm{Ind}^G_{B} ) \subset R(G)$ is a free submodule of rank $(p+5)/4$ generated by
\[
\left\{
\mathrm{Ind}^G_{B} \rho^{(\frac{p+1}{2})}_i
\middle|
0 \leq i \leq (p+1)/4
\right\}.
\]

Moreover, 
in both cases (i) and (ii), the following equality holds for any $i \in \Z/((p+1)/2)$: 
\begin{equation*}
\mathrm{Ind}^G_{B} \rho^{(\frac{p+1}{2})}_i = \mathrm{Ind}^G_{B} \rho^{(\frac{p+1}{2})}_{-i}.
\end{equation*}
\end{Lemma}
Lemma \ref{LEM:IndB} can be shown in a similar fashion to the proof of Lemma \ref{LEM:IndT}.

\begin{Lemma}\label{LEM:IndU}
Let $Q_p \subset \F_p^\times$ be the subgroup consisting of squares in $\F_p^\times$.
Then, $\mathrm{Im}( \mathrm{Ind}^G_{U} )$ is a free submodule of rank $3$ generated by 
$
\{ \mathrm{Ind}^G_U \chi^{(p)}_0, \mathrm{Ind}^G_U \chi^{(p)}_1, \mathrm{Ind}^G_U \chi^{(p)}_\Delta \}.
$
Furthermore, 
for distinct $i, j\in \F_p^\times$, 
$\mathrm{Ind}^G_U \chi^{(p)}_i = \mathrm{Ind}^G_U \chi^{(p)}_j$ 
if and only if 
either $\{ i, j\} \subset Q_p$ or $\{ i, j\} \subset \F_p^{\times} \setminus Q_p$.
\end{Lemma}
\begin{proof}
Let $s_0$ be a generator of $Q_p$ and $\varepsilon = (-1)^{(p-1)/2}$.
The conjugacy classes of $U$ under $\PSL_2(\F_p)$ are
\[
\{I_2\}, \quad 
\{ \sigma_{U}(s_0^j) \mid 0 \leq j \leq \frac{p-3}{2}\}, \quad 
\{ \sigma_{U}(s_0^j \Delta) \mid 0 \leq j \leq \frac{p-3}{2}\}.
\]
Thus, for $i\in \Z/p$, we obtain
\begin{equation*}
(\mathrm{Ind}^G_U \chi^{(p)}_i)(I_2)
= \dfrac{p^2-1}{2},\\[8pt]
\end{equation*}
\begin{equation*}
(\mathrm{Ind}^G_U \chi^{(p)}_i)(\sigma_{U}(s_0))
= \left\{
\begin{array}{ll}
\frac{p-1}{2}, & \text{ if }i = 0, \\[8pt]
\frac{-1+\sqrt{\varepsilon p}}{2}, & \text{ if }i \in Q_p, \\[8pt]
\frac{-1-\sqrt{\varepsilon p}}{2}, & \text{ otherwise, }
\end{array}
\right.\\[8pt]
(\mathrm{Ind}^G_U \chi^{(p)}_i)(\sigma_{U}(s_0\Delta))
= \left\{
\begin{array}{ll}
\frac{p-1}{2}, & \text{ if }i = 0, \\[8pt]
\frac{-1-\sqrt{\varepsilon p}}{2}, & \text{ if }i \in Q_p, \\[8pt]
\frac{-1+\sqrt{\varepsilon p}}{2}, & \text{ otherwise, }
\end{array}
\right.
\end{equation*}
and otherwise $\mathrm{Ind}^G_U \chi^{(p)}_i$ is $0$.
A careful verification of these computations allows us to deduce the claims  of Lemma \ref{LEM:IndU}. 
\end{proof}

\subsection{Result; $\alpha$-KDW invariants of lens spaces with $\PSL_2(\F_p)$}\label{PSLlens}
In this subsection, we will describe the $\alpha$-KDW invariants of $M=L^1(k;\ell)$ as the image of a certain homomorphism from the $\alpha$-KDW invariants with $\Z/k$. 
Hereafter, $k$ is an odd prime and $\ell \in \N$ is coprime to $k$.

Let us prepare the cyclic subgroup $H_k \subset G$ of order $k$ by 
\begin{equation}\label{cyc}
H_k\coloneqq
\left\{
\begin{split}
&\langle \sigma_{T}( \zeta_-^{\frac{p-1}{2 k}} ) \rangle, &&\text{ if $k | (p-1)$,}\\
&\langle \sigma_{U}( 1 ) \rangle, &&\text{ if $k=p$,}\\
&\langle \sigma_{B}( \zeta_+^{\frac{p+1}{2 k}} ) \rangle, &&\text{ if $k | (p+1)$.}\\
\end{split}
\right.
\end{equation}
Hereafter, we consider $H_k$ to be $\Z/k$ with the generator $\sigma_{T}(\zeta_-^{\frac{p-1}{2k}})$, $\sigma_{U}(1)$, or $\sigma_{B}(\zeta_+^{\frac{p+1}{2k}})$.
These generators are referred to as \emph{canonical generators}.
Note that $H_k$ includes all the representatives of order $k$ appearing in Tables \ref{TAB:rep1} and \ref{TAB:rep2}.

For a non-trivial homomorphism $\tilde{\varphi} : \pi_1(M) \to G$,
we will show that $\psi_G \circ B\tilde{\varphi}_* \circ \iota_{\pi_1 (M)} ( [M]_s )$ is equal to the image of the $G$-signature of a lens space under the $\mathrm{Ind}_{H_k}^G$.
The image $\tilde{\varphi}(\gamma)$ is conjugate to one of the representatives in Tables \ref{TAB:rep1} and \ref{TAB:rep2}, where $\gamma \in \pi_1(M)$ is the canonical generator in Example \ref{oo122}.
Since $\pi_1(M) \cong \Z/k$, there exist $g \in G$ and a homomorphism $\varphi:\pi_1 (M) \to H_k $ such that
\[\raisebox{-0.5\height}{\includegraphics[scale=1]{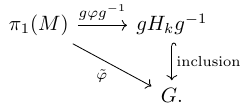}}\]
By (\ref{EQ:Ind}), we have
\begin{equation}\label{EQ:u1}
\begin{split}
(\psi_G \circ B\tilde{\varphi}_*) \circ \iota_{\pi_1 (M)} ( [M]_s )&= 
(\mathrm{Ind}^G_{g{H_k}g^{-1}} \circ \psi_{g{H_k}g^{-1}} \circ B(g\varphi g^{-1})_*) \circ \iota_{\pi_1 (M)}( [M]_s ) \\ &= 
\mathrm{Ind}^G_{{H_k}} \circ \psi_{{H_k}} \circ B\varphi_* \circ \iota_{\pi_1 (M)}( [M]_s ) \in \mathrm{Im}\,\mathrm{Ind}^G_{{H_k}} \otimes \Z_{(2)}/\Z. 
\end{split}
\end{equation}
In fact, for $r \in \Z/n$, the induced representation $\mathrm{Ind}_{\Z/m}^{\Z/nm} : R(\Z/n) \to R(\Z/nm)$ is described by
$
\mathrm{Ind}_{\Z/n}^{\Z/nm} \rho^{(n)}_{r} = \sum_{i=0}^{m-1} \rho^{(nm)}_{ni+r}.
$
Hence, via the inclusion from $H_k$ into $G$, 
Lemmas \ref{LEM:IndT}, \ref{LEM:IndB}, and \ref{LEM:IndU} show that $\mathrm{Ker} \, \mathrm{Ind}^G_{H_k}$ is additively generated by the following sets: 

\noindent
(i) if $k\neq p$,
\[
\left\{ \rho^{(k)}_{i} - \rho^{(k)}_{-i} \in R(\Z/k) \middle| i \in \Z/k \right\}.
\]

\noindent
(ii) if $k=p$,
\[
\left\{ \rho^{(p)}_{1} - \rho^{(p)}_{s} \in R(\Z/p) \middle| s \in Q_p \right\} 
\sqcup 
\left\{ \rho^{(p)}_{\Delta} - \rho^{(p)}_{s \Delta} \in R(\Z/p) \middle| s \in Q_p \right\}.
\]
Hereafter, the submodule $\mathrm{Ker} \, \mathrm{Ind}^G_{H_k}$ is denoted by $\mathcal{K}_k$.
Thus, by the homomorphism theorem, $\mathrm{Ind}^G_{H_k} : R(H_k) \to R(G)$ induces an additive isomorphism
\[
\overline{\mathrm{Ind}}^G_{H_k} : R(H_k)/\mathcal{K}_k \xrightarrow{\quad\sim\quad} \mathrm{Im}\,\mathrm{Ind}^G_{H_k}.
\]
Then, the right-hand side of (\ref{EQ:u1}) is equal to the image of $\overline{\mathrm{Ind}}^G_{H_k}$ from 
\begin{equation}\label{EQ:u2}
[ \psi_{{H_k}} \circ B\varphi_* \circ \iota_{\pi_1(M)}( [M]_s )] \in R(H_k)/\mathcal{K}_k \otimes \Z_{(2)}/\Z .
\end{equation}
Here, for $x \in R(H_k)$, we denote by $[x] \in R(H_k)/\mathcal{K}_k$ the equivalence class of $x$.
By considering both $\pi_1(M)$ and $H_k$ to be $\Z/k$ with the canonical generators,
(\ref{EQ:psiG}) shows that (\ref{EQ:u2}) is equal to $[\varphi_! \circ \alpha^{\mathrm{rp}}_\Z (\Z/k,S^3)].$

Now, we canonically regard $\overline{\mathrm{Ind}}^G_{H_k} \otimes \mathrm{Id}_{\Z_{(2)}/\Z}$ as a homomorphism between group rings 
\[
\overline{\mathrm{Ind}}^G_{H_k} \otimes \mathrm{Id}_{\Z_{(2)}/\Z} : \Z [R(\Z/k)/\mathcal{K}_k \otimes \Z_{(2)}/\Z] \to \Z[R(G)\otimes \Z_{(2)}/\Z].
\]
Thus,
$\text{$\alpha$-}\mathrm{KDW}_G(M)$ is equal to
\[
\overline{\mathrm{Ind}}^G_{H_k} (
(1-m) \boldsymbol{o} + m \sum_{\varphi \in \mathrm{Hom}(\pi_1 (M), \Z/k)} 1_\Z [\varphi_! \circ \alpha^{\mathrm{rp}}_\Z (\Z/k,S^3)]) \in \Z [R(\Z/k)/\mathcal{K}_k \otimes \Z_{(2)}/\Z],
\]
where $m$ denotes the number of cyclic subgroups of order $k$ and $\boldsymbol{o} \in R(\Z/k)/\mathcal{K}_k \otimes \Z_{(2)}/\Z$ denotes the additive zero.
To summarize:

\begin{Proposition}\label{lenswithPSL}
Under the above notation,
$\text{$\alpha$-}\mathrm{KDW}_G(M)$ corresponds to the image of 
\[
(1-m_k) \boldsymbol{o} + m_k \,\mathcal{Q}(\alpha\mathchar`-\mathrm{KDW}_{\Z/k} (M))
\]
under $\overline{\mathrm{Ind}}^G_{H_k} \otimes \mathrm{Id}_{\Z_{(2)}/\Z}$,
where 
\[
m_k=
\left\{
\begin{split}
&(p^2+p)/2, &&\text{if $k|(p-1)$,}\\
&p+1, &&\text{if $k=p$,}\\
&(p^2-p)/2, &&\text{if $k|(p+1)$.}\\
\end{split}
\right.
\]
Here, $\mathcal{Q} : \Z[R(H_k)\otimes \Z_{(2)}/\Z] \to \Z[R(H_k) / \mathcal{K}_k \otimes \Z_{(2)}/\Z]$ is defined by the canonical extension of the quotient map $[\cdot] : R(H_k) \to R(H_k) / \mathcal{K}_k$.
\end{Proposition}
The following lemma immediately allows us to deduce Proposition \ref{lenswithPSL}.
\begin{Lemma}
Suppose that $k$ is an odd prime factor of $p^3-p$.

\rm{(i)} If $k |p - 1$,
then the number of the subgroups of $G$ of order $k$ is $(p^2+p)/2$.

\rm{(ii)} If $k |p + 1$,
then the number of the subgroups of $G$ of order $k$ is $(p^2-p)/2$.

\rm{(iii)} If $k = p$, 
then the number of the subgroups of $G$ of order $k$ is $p+1$.
\end{Lemma}
\begin{proof}
Since there is a $k$-Sylow subgroup of $G$ in the cyclic subgroups $T, U$, or $B$,
each of the $k$-Sylow subgroups of $G$ is cyclic.
Therefore, each $k$-Sylow subgroup contains exactly one cyclic subgroup of order $k$.
Furthermore, since each cyclic subgroup of order $k$ is contained in one of the $k$-Sylow subgroups and all the $k$-Sylow subgroups are mutually conjugate to each other, all the cyclic subgroups of order $k$ are mutually conjugate to each other.
Therefore, the number of cyclic subgroups of order $k$ in $G$ is equal to the index $(G: N_G(H_k))$.
Here, $N_G(H_k)$ is the normalizer of $H_k$ in $G$.
Now, let us compute the index $(G: N_G(H_k))$.

(i)
Suppose that $k |p - 1$, and
let $
s=
\begin{pmatrix}
0 & -1 \\
1 & 0 
\end{pmatrix}.
$
We note that $s \sigma_{T}(\zeta_-^{\frac{p-1}{2 k}}) s^{-1} = \sigma_{T}(\zeta_-^{(k-1) \frac{p-1}{2 k}})$.
Since $k\neq 2$, the conjugacy class in $H_k$ of the generator $\sigma_{T}(\zeta_-^{\frac{p-1}{2 k}})\in H_k$ over $N_G (H_k)$ is
\[
\{\sigma_{T}(\zeta_-^{\frac{p-1}{2 k}}), \sigma_{T}(\zeta_-^{(k-1) \frac{p-1}{2 k}}) \}.
\]
Therefore, for $g \in N_G(H_k)$, either $g \sigma_{T}(\zeta_-^{\frac{p-1}{2 k}}) g^{-1} = \sigma_{T}(\zeta_-^{\frac{p-1}{2 k}})$ or $g \sigma_{T}(\zeta_-^{\frac{p-1}{2 k}}) g^{-1} = \sigma_{T}(\zeta_-^{(k-1) \frac{p-1}{2 k}})$ must necessarily hold.
These equalities imply $N_G(H_k) \subset T \sqcup sT$.
Conversely, when $g \in T \sqcup sT$, we easily see $g\in N_G(H_k)$. 
Hence, $|N_G(H_k)| = |T \sqcup sT| = p-1$, which leads to $(G: N_G(H_k)) = (p^2+p)/2$.

(ii)
Suppose that $k |p - 1$.
Take $\widetilde{x}$ and $\widetilde{y} \in \F_p$ satisfying $\widetilde{x}^2 - \Delta \widetilde{y}^2 = -1$ and let
\[
s' = \begin{pmatrix}
\widetilde{x} & \Delta \widetilde{y} \\
\widetilde{y} & \widetilde{x} 
\end{pmatrix}
\begin{pmatrix}
1 & 0 \\
0 & -1 
\end{pmatrix}.
\]
Then, by replacing $s$ with $s'$, claim (ii) follows in a similar manner to claim (i).

(iii) Suppose that $k=p$. 
Then, $H_k$ corresponds to $U$.
Since \[
\begin{pmatrix}
a & b \\
c & d 
\end{pmatrix}
\begin{pmatrix}
1 & x \\
0 & 1 
\end{pmatrix}
\begin{pmatrix}
d & -b \\
-c & a 
\end{pmatrix}
=
\begin{pmatrix}
1-acx & a^2 x \\
-c^2x & 1+acx 
\end{pmatrix},
\]
$g \in N_G(H_k)$ if and only if $g = \begin{pmatrix}
a & b \\
0 & a^{-1} 
\end{pmatrix}$ for some $a\in \F_p^\times$ and $b\in \F_p$.
Hence, $(G: N_G(H_k)) = p+1$.
\end{proof}

\section{$\alpha$-KDW invariants of Brieskorn homology spheres}\label{SEC:Brieskorn}
In this section, we compute the $\alpha$-KDW invariants of Brieskorn homology spheres.
Let us assume $G=\PSL_2 (\F_p)$ and fix $m \in \N$.
We also fix the pairwise coprime integers $k_1, k_2, \ldots , k_m$.

\subsection{Seifert homology spheres and Brieskorn homology spheres}\label{tSm}
Let us briefly review the Seifert homology spheres from \cite{JonWes}.
\emph{The Seifert homology sphere} $\Sigma = \Sigma(k_1, \ldots , k_m)$ (of genus $0$) is a spin$^{c}\,\,3$-manifold defined as the circle bundle over the orbifold of genus $0$, whose homology with $\Z$ coefficients is isomorphic to $H_*(S^3;\Z)$ (see, e.g., \cite{JonWes}).
In particular, if $m=3$, then $\Sigma$ is called \emph{a Brieskorn homology sphere}.

$\pi_1(\Sigma)$ has the group presentation
\begin{equation}
\widetilde{\Gamma} \coloneqq
\langle h,x_1,\ldots, x_m | 
x_i h x_i^{-1} h^{-1}=1, 
x_i^{k_i} h^{\ell_i} =
x_1 x_2 \cdots x_m = 1,
\text{ for $i \in \{ 1, \ldots , m\}$}
\rangle,
\end{equation}
where $\ell_1, \ell_2, \ldots , \ell_m$ are some integers satisfying
\begin{equation}\label{hom}
k_1 \cdots k_m \sum_{i=1}^{m} \dfrac{\ell_i}{k_i} \in \{ \pm 1 \}.
\end{equation}
Let $\Gamma$ be the quotient group of $\widetilde{\Gamma}$ subject to its center $\langle h \rangle$.
That is,
\begin{equation}
\Gamma \coloneqq \langle x_1,\ldots, x_m | 
x_i^{k_i} = 
x_1 x_2 \cdots x_m = 1,
\text{ for $i \in \{ 1, \ldots , m\}$}
\rangle. 
\end{equation}
Define $\mathcal{L}$ to be the disjoint union $L^1(k_1,\ell_1) \sqcup \cdots \sqcup L^1(k_m,\ell_m)$.
According to \cite{JonWes}, by appropriately choosing the spin$^{c}$ structure of $\mathcal{L}$, we can take a spin$^{c}$ bordism $W$ between $\Sigma$ and $\mathcal{L}$.
The bordism $W$ satisfies the following three conditions:
\begin{enumerate}[(i)]
\item The boundary of $W$ is the disjoint union $\Sigma \sqcup \mathcal{L}$, and the fundamental group $\pi_1(W)$ is isomorphic to $\Gamma$.

\item The inclusion $\Sigma \subset W$ induces a homomorphism $\pi_1 (\Sigma) \to \pi_1(W)$ which coincides with the quotient homomorphism $\widetilde{\Gamma} \to \Gamma$.

\item For any $1 \leq i \leq m$, the homomorphism $\pi_1(L^1(k_i,\ell_i)) \to \pi_1(W)$ induced by the inclusion $L^1 (k_i; \ell_i)\subset W$ sends a canonical generator of $\pi_1(L^1(k_i,\ell_i))$ to $x_i$. 
Here, the canonical generator of $\pi_1(L^1(k_i,\ell_i))$ is defined in Example \ref{oo122}.
\end{enumerate}

Furthermore, we prepare a lemma.
\begin{Lemma}\label{center}
Let $J$ be a group such that every perfect subgroup of $J$ is centerless, and let $1_J$ denote the identity of $J$.
Then, any homomorphism $f: \widetilde{\Gamma} \to J$ satisfies $f(h)=1_J$.
In particular, any homomorphism $f: \widetilde{\Gamma} \to J$ factors through $\Gamma$. 
\end{Lemma}
\begin{proof}
Since $H_1(M;\Z) = 0$, $\widetilde{\Gamma}$ is perfect.
Hence, $\mathrm{Im} f$ is also perfect, and therefore  is centerless.
Since $f(h)$ belongs to the center of $\mathrm{Im} f$, it follows that $f(h)=1_J$.
\end{proof}

\begin{Example}\label{sur}
For $p \geq 5$, $\PSL_2(\F_p)$ is an example of a group such that all perfect subgroups are centerless \cite{Hup,Luc}.
Explicitly, the proper perfect subgroups of $\PSL_2(\F_p)$ uniquely exist if $p^2 \equiv 1 \mod 5$, and is isomorphic to the alternating group $A_5$.
Especially, if $m \geq 3$ and $k_1,\ldots,k_m$ are mutually distinct odd primes, any homomorphism $f:\widetilde{\Gamma}\to\PSL_2(\F_p)$ is either surjective or trivial since $|A_5|=60$ and at least one of $k_1,\ldots,k_m$ is greater than $5$.
\end{Example}

\subsection{Resulting computation}\label{SEC:comp}
Let us state Theorems \ref{withoutp} and \ref{withp}, and illustrate examples.
Let $m=3$, and $k_1,k_2,k_3$ be mutually distinct odd primes dividing $|G|$.
Recall the notations $P_k^{(m)}$ and $\boldsymbol{\xi}(k;\ell)$ from Example \ref{oo122}, and $H_k, \overline{\mathrm{Ind}}^G_{H_k}, \mathcal{K}_k,$ and $[\cdot]$ from Section \ref{PSLlens}.
We denote by $\boldsymbol{o}$ the additive zero.
Moreover, let $\bigoplus_{i=1}^{3} \overline{\mathrm{Ind}}^G_{H_{k_i}}: \bigoplus_{i=1}^{3}R(H_{k_i})/\mathcal{K}_{k_i} \to R(G)$ be the map defined by 
\[
\bigoplus_{i=1}^{3} \overline{\mathrm{Ind}}^G_{H_{k_i}} ( \varphi_1, \varphi_2, \varphi_3)
\coloneqq
\sum_{i=1}^{3} \overline{\mathrm{Ind}}^G_{H_{k_i}} ( \varphi_i )
\]
for representations $\varphi_1\in R(H_{k_1}), \varphi_2\in R(H_{k_2}),$ and $\varphi_3 \in R(H_{k_3})$,
and regard $\bigoplus_{i=1}^{3} \overline{\mathrm{Ind}}^G_{H_{k_i}} \otimes \mathrm{Id}_{\Z_{(2)}/\Z}$ as a homomorphism
$
\Z[\bigoplus_{i=1}^{3}R(H_{k_i})/\mathcal{K}_{k_i}\otimes \Z_{(2)}/\Z] \to \Z[R(G) \otimes \Z_{(2)}/\Z].
$

\begin{Theorem}\label{withoutp}
Suppose that $p \notin \{ k_1,k_2,k_3 \}$.
Then, $\text{$\alpha$-KDW}_G (\Sigma)$ is equal to the image of 
\begin{equation}\label{EQ:withoutp}
\boldsymbol{o}+
4 |G| \sum_{m_1=1}^{(k_1-1)/2} \, \sum_{m_2=1}^{(k_2-1)/2} \, \sum_{m_3=1}^{(k_3-1)/2}
1_\Z (
[P_{k_1}^{(m_1)} \boldsymbol{\xi}(k_1;\ell_1)], \,\,
[P_{k_2}^{(m_2)} \boldsymbol{\xi}(k_2;\ell_2)], \,\,
[P_{k_3}^{(m_3)} \boldsymbol{\xi}(k_3;\ell_3)]
)
\end{equation}
under the homomorphism $\bigoplus_{i=1}^{3} \overline{\mathrm{Ind}}^G_{H_{k_i}} \otimes \mathrm{Id}_{\Z_{(2)}/\Z}$.
\end{Theorem}

\begin{Theorem}\label{withp}
Suppose $k_1=p$.
Then, $\text{$\alpha$-KDW}_G (\Sigma)$ is equal to the image of 
\begin{equation}\label{EQ:withp}
\begin{split}
\boldsymbol{o}
&+ 2|G| \sum_{m_2=1}^{(k_2-1)/2} \, \sum_{m_3=1}^{(k_3-1)/2}
1_\Z (
[P_{p}^{(1)} \boldsymbol{\xi}(p;\ell_1)], \,\,
[P_{k_2}^{(m_2)} \boldsymbol{\xi}(k_2;\ell_2)], \,\,
[P_{k_3}^{(m_3)} \boldsymbol{\xi}(k_3;\ell_3)] 
)\\ \,\,\,\,
&+2|G| \sum_{m_2=1}^{(k_2-1)/2} \, \sum_{m_3=1}^{(k_3-1)/2}
1_\Z (
[P_{p}^{(\Delta)} \boldsymbol{\xi}(p;\ell_1)], \,\,
[P_{k_2}^{(m_2)} \boldsymbol{\xi}(k_2;\ell_2)], \,\,
[P_{k_3}^{(m_3)} \boldsymbol{\xi}(k_3;\ell_3)] 
)
\end{split}
\end{equation}
under the homomorphism $\bigoplus_{i=1}^{3} \overline{\mathrm{Ind}}^G_{H_{k_i}} \otimes \mathrm{Id}_{\Z_{(2)}/\Z}$.
\end{Theorem}
The proofs are shown in Section \ref{SEC:proofmainthm}. 

\begin{Example}
Suppose that $p=29$, $(k_1,k_2,k_3)=(3,5,7)$, and $(\ell_1,\ell_2,\ell_3)=(2,1,-6)$.
Then, $p \notin \{ k_1,k_2,k_3 \}$, and (\ref{EQ:withoutp}) in Theorem \ref{withoutp} is computed as
\[ 
\boldsymbol{o}+
48720 \sum_{1 \leq i \leq 2, \,1 \leq j \leq 3} (\boldsymbol{u}_{11},\boldsymbol{u}_{2i},\boldsymbol{u}_{3j}).
\]
Here, 
\[
\boldsymbol{u}_{11}=\frac{1}{3}
\begin{pmatrix}
1 \\
2 
\end{pmatrix}, \quad
\boldsymbol{u}_{21}=\frac{1}{5}
\begin{pmatrix}
4 \\
2 \\
4 
\end{pmatrix}, \quad
\boldsymbol{u}_{22}=\frac{1}{5}
\begin{pmatrix}
4 \\
4 \\
2 
\end{pmatrix},
\]
\[
\boldsymbol{u}_{31}=\frac{1}{7}
\begin{pmatrix}
3 \\
3 \\
1 \\
0 
\end{pmatrix}, \quad
\boldsymbol{u}_{32}=\frac{1}{7}
\begin{pmatrix}
3 \\
0 \\
3 \\
1 
\end{pmatrix}, \quad
\boldsymbol{u}_{33}=\frac{1}{7}
\begin{pmatrix}
3 \\
1 \\
0 \\
3 
\end{pmatrix}. \quad
\]
\end{Example}

\begin{Example}
Suppose that $p=11$, $(k_1,k_2,k_3)=(p,3,5)$, and $(\ell_1,\ell_2,\ell_3)=(3,1,-3)$.
Then, (\ref{EQ:withp}) in Theorem \ref{withp} is presented by
\[
\boldsymbol{o} + 
2640( 
(\boldsymbol{u}_{11},\boldsymbol{u}_{21},\boldsymbol{u}_{31}) + 
(\boldsymbol{u}_{11},\boldsymbol{u}_{21},\boldsymbol{u}_{32}) 
).
\]
Here, 
\[
\boldsymbol{u}_{11}=
\frac{1}{11}
\begin{pmatrix}
6 \\
8 \\
8 
\end{pmatrix},\quad
\boldsymbol{u}_{21}=
\frac{1}{9}
\begin{pmatrix}
2 \\
7 
\end{pmatrix},\quad
\boldsymbol{u}_{31}=
\frac{1}{5}
\begin{pmatrix}
0 \\
2 \\
3 
\end{pmatrix},\quad
\boldsymbol{u}_{32}=
\frac{1}{5}
\begin{pmatrix}
0 \\
3 \\
2 
\end{pmatrix}.
\]
\end{Example}

\subsection{$\Hom (\pi_1(\Sigma), \PSL_2 (\F_p))$ and the Chebyshev polynomials}\label{TRG}
In this subsection, we analyze  $\Hom(\pi_1(\Sigma),G)$ in order to prove Theorems \ref{withoutp} and \ref{withp}.
The proofs of the lemmas and the proposition in this subsection can be found in \ref{SEC:proofBrieskorn}.

As in Section \ref{SEC:comp}, let $m=3$ and $k_1,k_2,$ and $k_3$ be distinct odd primes dividing $|G|$.
Let us define $\overline{\mathrm{Tr}}: G \to \F_p/\{ \pm 1 \}$ by $\overline{\mathrm{Tr}} X \coloneqq \mathrm{Tr} x$, where $x \in \SL_2(\F_p)$ is a representative of $X \in G$.
First, we recall the definition and properties of the Chebyshev polynomial.
\emph{The (normalized) Chebyshev polynomial $S_n(t) \in \F_p[t]$} is defined by
\begin{equation}
S_1(t) = 1, \qquad S_2(t)=t, \qquad S_{n+1} (t) = t S_{n}(t) - S_{n-1}(t).
\end{equation}
By the Cayley-Hamilton theorem,  the following holds for any $X \in \SL_2 (\F_p)$:
\begin{equation}\label{OrderCheb}
X^n = S_n(\mathrm{Tr} X) X - S_{n-1}(\mathrm{Tr} X) I_2.
\end{equation}
\begin{Lemma}\label{aab}
Let $X \in G \setminus \{ I_2\}$ and $n\in \N \setminus \{ 1 \}$ be odd.
Then, $X^n = I_2$ if and only if 
\begin{equation}\label{EQ:aab}
S_{2n}(\overline{\mathrm{Tr}} X) =S_{2n-1}(\overline{\mathrm{Tr}} X)+1= 0 \in \F_p/\{ \pm 1\}.
\end{equation}
\end{Lemma}

Let us discuss the solutions of (\ref{EQ:aab}).
By an invariance of $\overline{\mathrm{Tr}}$ under conjugation, 
Lemma \ref{aab} shows that the solutions are described by the traces of some representatives of conjugacy classes in $G$.
Let us recall $H_k$ defined in (\ref{cyc}).
Since the subgroup $H_k$ contains all the representatives of order $k$ shown in Tables \ref{TAB:rep1} and \ref{TAB:rep2}, the following lemma is obvious.

\begin{Lemma}\label{ansPSL}
Let $k$ be an odd prime dividing $|G|$. 
Define
\[
\mathrm{Sol}_k \coloneqq \left\{ t \in \F_p/ (\Z/2) \,\, \middle| S_{2k}(t) = S_{2k-1}(t) +1 =0 \right\}.
\]

\noindent
(i)
$\mathrm{Sol}_k = \overline{\mathrm{Tr}} ( H_k \setminus \{ I_2 \} )$. 
In particular, $0 \notin \mathrm{Sol}_k$.

\noindent
(ii)
When $k$ is not equal to $p$, either $X=Y$ or $X = Y^{-1}$ if and only if $\overline{\mathrm{Tr}}X = \overline{\mathrm{Tr}}Y$ for $X,Y \in H_k \setminus \{ I_2 \}$.
Therefore, the induced map $\overline{\mathrm{Tr}} : (H_k \setminus \{ I_2 \}) / \mathrm{conj} \to \mathrm{Sol}_k$ is bijective.

\noindent
Here, $(H_k \setminus \{ I_2 \}) / \mathrm{conj}$ is the quotient set subject to conjugacy of the normalizer $N_G (H_{k})$.
\end{Lemma}

Now, using the Chebyshev polynomials, we describe a relationship between $\Hom (\Gamma, G)$ and the Chebyshev polynomials.
Define $\widetilde{\mathcal{A}}$ to be
\begin{equation}
\left\{ 
(a,b,c)\in (\F_p)^3
\middle| \,\,\,\,\,
\begin{aligned}
& S_{2 k_1} (a) = S_{2 k_2} (b) = S_{2 k_3} (c) = 0,\\
& S_{2 k_1-1} (a) = S_{2 k_2-1} (b) = S_{2 k_3-1} (b) =-1
\end{aligned}
\right\}.
\end{equation}
Since $S_{2k}(-t) = - S_{2k}(t)$ and $S_{2k-1}(-t) = S_{2k-1}(t)$, 
it follows that $(a,b,c) \in \widetilde{\mathcal{A}}$ implies $(\varepsilon_1 a,\varepsilon_2 b,\varepsilon_3 c) \in \widetilde{\mathcal{A}}$ for any $\varepsilon_i \in \{\pm 1 \}$.
Let us define an equivalence relation $\sim$ on $\widetilde{\mathcal{A}}$ by
\[
(a,b,c) \sim (\varepsilon_1 a,\varepsilon_2 b,\varepsilon_3 c) \text{ if $\varepsilon_i \in \{ 1 , -1 \}$ and $\varepsilon_1\varepsilon_2\varepsilon_3 = 1,$ }
\]
and denote by $\mathcal{A}$ the quotient set $\widetilde{\mathcal{A}}/\sim$.
We define a map 
\[
\mathcal{F}: \Hom (\Gamma, G) \setminus \{ e_G \} \to \mathcal{A}
\]
by 
$\mathcal{F}(f) \coloneqq (\mathrm{Tr}X_1,\mathrm{Tr}X_2,\mathrm{Tr}(X_2^{-1} X_1^{-1}))$, where $X_1,X_2\in\SL_2(\F_p)$ are some representatives of $f(x_1),f(x_2)\in \PSL_2(\F_p)$, respectively.
Here, $e_G \in \Hom (\Gamma, G)$ denotes the trivial homomorphism.
Furthermore, $\mathcal{F}$ canonically induces $\overline{\mathcal{F}} : (\Hom (\Gamma, G) \setminus \{ e_G \})/\mathrm{conj}\to \mathcal{A}$.

\begin{Lemma}\label{LEM:surjective}
$\mathcal{F}$ is surjective.
\end{Lemma}
\begin{Lemma}\label{numc}
Let $(a,b,c) \in A$.
Then, there exist two homomorphisms $\varphi_1, \varphi_2 \in \mathcal{F}^{-1}(a,b,c)$ satisfying the following three conditions:

\noindent
\rm{(i)}
$\varphi_1$ and $\varphi_2$ are not conjugate.

\noindent
\rm{(ii)}
Every homomorphism in $\mathcal{F}^{-1}(a,b,c) \setminus \{ \varphi_1,\varphi_2\}$ is conjugate to either $\varphi_1$ or $\varphi_2$.

\noindent
\rm{(iii)}
$\varphi_i(x_1)$ is equal to some representatives in Tables \ref{TAB:rep1} and \ref{TAB:rep2} for $i \in \{ 1, 2 \}$.
In particular, if $k_1 = p$, then $\{\varphi_1(x_1), \varphi_2(x_1) \} = \{ \sigma_U(1), \sigma_U(\Delta) \}$.

In particular, the induced map $\overline{\mathcal{F}}$ is a $2$-$1$ map.
\end{Lemma}
\begin{Lemma}\label{numd}
Take distinct $(a,b,c), (a',b',c') \in \mathcal{A}$.
Then, the elements of $\mathcal{F}^{-1}(a,b,c)$ and those of $\mathcal{F}^{-1}(a',b',c')$ are not conjugate to each other.
\end{Lemma}

Consequently, we can describe $\left| \, \Hom(\Gamma, \PSL_2(\F_p)) \, \right|$ by using $\mathcal{A}$.
\begin{Proposition}\label{PROP:order}
Let $k_1,k_2,$ and $k_3$ be distinct odd primes, which divide  $|G|$.
Then, 
\begin{equation}\label{EQ:PROP:order}
\left| \,\Hom(\Gamma, G) / \mathrm{conj} \, \right| = 2 \left| \, \mathcal{A} \, \right| + 1, \qquad
\left| \, \Hom(\Gamma, G) \, \right|=(p^3-p)\left|\,\mathcal{A} \, \right| + 1.
\end{equation}
\end{Proposition}
\begin{Remark}
\cite{Mar} has computed the cardinality of $\Hom(\Gamma, G)$ for more general primes $p, k_1, k_2,$ and $k_3$, using the Frobenius formula.
However, the conjugacy classes of $\Hom(\Gamma, G)$ were not concretely given.
In this work, in order to compute the $\alpha$-KDW invariants, we provide the conjugacy classes of $\Hom(\Gamma, G)$ concretely.
\end{Remark}

\section{Proof of Theorems \ref{withoutp} and \ref{withp}}\label{SEC:proofmainthm}
We now prove Theorems \ref{withoutp} and \ref{withp}.
Throughout this section, let $L_i$ be $L^1 (k_i;\ell_i)$ for $i\in\{ 1,2,3 \}$.
Fix $(a,b,c) \in \mathcal{A}$, and take $\varphi_1,\varphi_2 \in \mathcal{F}^{-1} (a,b,c)$ satisfying the three conditions in Lemma \ref{numc}. 

We examine $\psi_G \circ (B\varphi_j \circ\iota_{\pi_1(M)})_* ([\Sigma]_s)$.
For $j\in\{ 1,2 \}$, we define a homomorphism
\[
\widetilde{\varphi}_{ji} : \pi_1(L_i) \longrightarrow G: \qquad \gamma_i \longmapsto \varphi_j (x_i).
\]
Here, we denote by $\gamma_i \in \pi_1(L_i)$ the canonical generator of $\pi_1(L_i)$ in Example \ref{oo122}.
Then, the bordism $W$ between $\Sigma$ and $\mathcal{L}$ in Section \ref{tSm} implies 
\begin{eqnarray}
\psi_G \circ (B\varphi_j \circ\iota_{\pi_1(M)})_* ([\Sigma]_s) &=& \sum_{i=1}^{3} \psi_G \circ (B\widetilde{\varphi}_{ji} \circ\iota_{\pi_1(L_i)})_* ([L_i]_s) \in K_1(BG).
\end{eqnarray}
As in Section \ref{PSLlens}, consider the homomorphisms $\varphi_{ji}: \pi_1 (L_i) \to H_{k_i}$ satisfying a commutative diagram
\[\raisebox{-0.5\height}{\includegraphics[scale=1]{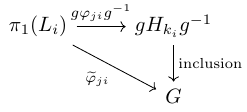}}\]
for some $g\in G$.
By an argument similar to that in Section \ref{PSLlens}, we have
\[
\psi_G \circ (B\widetilde{\varphi}_{ji} \circ\iota_{\pi_1(L_i)})_* ([L_i]_s)
=
\overline{\mathrm{Ind}}^G_{H_{k_i}} ( [(\varphi_{ji})_! \alpha^{\mathrm{rp}}_{\Z} (\Z/k_i, S^3)] ).
\]
Let us define 
\[
\alpha_j^{(a,b,c)} \coloneqq 
( [(\varphi_{j1})_! \alpha^{\mathrm{rp}}_{\Z} (\Z/k_1, S^3)],\,\,\,
[(\varphi_{j2})_! \alpha^{\mathrm{rp}}_{\Z} (\Z/k_2, S^3)],\,\,\,
[(\varphi_{j3})_! \alpha^{\mathrm{rp}}_{\Z} (\Z/k_3, S^3)] \,\,\,).
\]
Then, the equality
\[
\psi_G \circ (B\varphi_{j}\circ\iota_{\pi_1(M)})_* ([M]_s) 
= 
\bigoplus_{i=1}^3 \mathrm{\overline{Ind}}^G_{H_{k_i}} (\alpha_j^{(a,b,c)})
\]
holds for any $(a,b,c) \in \mathcal{A}$.

Thus, $\text{$\alpha$-KDW}_G (\Sigma)$ is the image of 
\begin{equation}\label{EQ:imageofalphaKDW}
\boldsymbol{o} + |G| \sum_{(a,b,c) \in \mathcal{A} } 1_\Z (\alpha_1^{(a,b,c)}) + 1_\Z(\alpha_2^{(a,b,c)}) \in \Z [\bigoplus_{i=1}^3R(H_{k_i})/\mathcal{K}_{k_i}\otimes \Z_{(2)}/\Z ],
\end{equation}
under the homomorphism 
\[
\bigoplus_{i=1}^3 \mathrm{\overline{Ind}}^G_{H_{k_i}} : 
\Z [\bigoplus_{i=1}^3R(H_{k_i})/\mathcal{K}_{k_i}\otimes \Z_{(2)}/\Z ]
\longrightarrow
\Z \left[R(G)\otimes \Z_{(2)}/\Z \right].
\]

Hereafter, 
for an odd prime $k \neq p$ dividing $p^2-1$ and $h \in \Z/k \setminus \{ 0 \}$, 
we also define a homomorphism $f^{(h)}: \pi_1( L_i ) \to H_{k_i}$ by
\[
f^{(h)}(\gamma_i) = 
\left\{\begin{split}
\sigma_{T}(\zeta_-^{h\frac{p-1}{2 k_i}}), &\text{ if $k_i|(p-1)$, } \\
\sigma_{B}(\zeta_+^{h\frac{p+1}{2 k_i}}), &\text{ if $k_i|(p+1)$, }
\end{split}
\right.
\]
as in Example \ref{oo122}.

\subsection{Case with $p \notin \{ k_1,k_2,k_3 \}$}
Let us prove Theorem \ref{withoutp}.
Suppose $p \notin \{ k_1,k_2,k_3 \}$.
Let $[\! [ k_i ]\! ]$ be $\{ 1, 2, \ldots , (k_i-1)/2\}$.
Define a bijective map $\mathcal{I}_i: [\! [ k_i ]\! ] \to (H_{k_i} \setminus \{ I_2 \})/ \mathrm{conj}$ by 
\[
\mathcal{I}_i(m)
=
\left\{
\begin{split}
\sigma_T (\zeta_-^{m \frac{p-1}{2 k_i}}), \text{ if $k_i | (p-1)$,} \\
\sigma_B (\zeta_+^{m \frac{p+1}{2 k_i}}),\text{ if $k_i | (p+1)$,}
\end{split}
\right.
\]
and, a map  $\mathcal{J}: (\Hom(\Gamma,G) \setminus \{ e_G\} )/\mathrm{conj} \to \prod_{i=1}^{3} (H_{k_i} \setminus \{ I_2 \})/ \mathrm{conj}$ by
$\mathcal{J}(\varphi_j) = (\varphi_{j1}(1),\varphi_{j2}(1),\varphi_{j3}(1))$.
Owing to the equation $\Hom(\Gamma,G) \setminus \{ e_G\} = \bigsqcup_{(a,b,c) \in \mathcal{A} }\mathcal{F}^{-1} (a,b,c)$ and Lemma \ref{numc}, we can define $\mathcal{J}$ for any element in $(\Hom(\Gamma,G) \setminus \{ e_G\} )/\mathrm{conj}$.
Then, the following commutative diagram holds:
\[\raisebox{-0.5\height}{\includegraphics[scale=1]{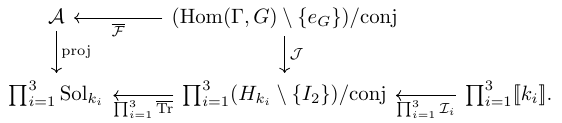}}\]
Lemma \ref{ansPSL} implies that $\prod_{i=1}^3 \overline{\mathrm{Tr}}$ is a bijection.
Let $\mathcal{P} = (\prod_{i=1}^3 \mathcal{I}_i)^{-1} \circ (\prod_{i=1}^3 \overline{\mathrm{Tr}})^{-1} \circ \mathrm{proj}$.
Then, $\mathcal{P}$ is a $2$-$1$ surjective map such that $\mathcal{P} (a,b,c) = \mathcal{P} (a,b,-c)$.
If $\mathcal{J}(\varphi_j) = (\mathcal{I}_1 \times \mathcal{I}_2 \times \mathcal{I}_3)(m_1,m_2,m_3)$, then we have $\varphi_{ji}$ and $f^{(m_i)}$ are conjugate under $N_G (H_{k_i})$.
Hence, if $\overline{\mathcal{F}}(\varphi_j)=(a,b,c)$ and $\mathcal{P}(a,b,c)=(m_1,m_2,m_3)$, then
\begin{align*}
\alpha_j^{(a,b,c)} 
&=
(
[(\varphi_{j1})_! \alpha^{\mathrm{rp}}_{\Z} (\Z/k_1, S^3)], \,\,
[(\varphi_{j2})_! \alpha^{\mathrm{rp}}_{\Z} (\Z/k_2, S^3)], \,\,
[(\varphi_{j3})_! \alpha^{\mathrm{rp}}_{\Z} (\Z/k_3, S^3)]
) \\
&=
(
[f^{(m_1)}_! \alpha^{\mathrm{rp}}_{\Z} (\Z/k_1, S^3)], \,\,
[f^{(m_2)}_! \alpha^{\mathrm{rp}}_{\Z} (\Z/k_2, S^3)], \,\,
[f^{(m_3)}_! \alpha^{\mathrm{rp}}_{\Z} (\Z/k_3, S^3)]
) \\
&=
(
[P_{k_1}^{(m_1)} \boldsymbol{\xi}(k_1;\ell_1)], \,\,
[P_{k_2}^{(m_2)} \boldsymbol{\xi}(k_2;\ell_2)], \,\,
[P_{k_3}^{(m_3)} \boldsymbol{\xi}(k_3;\ell_3)]
).
\end{align*}
Hence, applying this result to $\alpha_j^{(a,b,c)}$ in (\ref{EQ:imageofalphaKDW}) gives (\ref{EQ:withoutp}).\qed

\subsection{Case with $p \notin \{ k_1,k_2,k_3 \}$}
Let us prove Theorem \ref{withp}.
Since $\Sigma(k_1,k_2,k_3) = \Sigma(k_2,k_3,k_1) = \Sigma(k_3,k_1,k_2)$, we may assume $k_1=p$.
Note that, if $k_1 = p$ and $(a,b,c)\in\mathcal{A}$, then $a\in\{2,-2\}$.
Consider a diagram
\begin{equation}\label{EQ:diagramA}\raisebox{-0.5\height}{\includegraphics[scale=1]{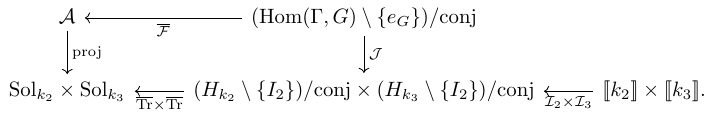}}\end{equation}
Here, the maps $\mathrm{proj}$ and $\mathcal{J}$ are defined by 
\[
\mathrm{proj}(a,b,c)=(b,c), \qquad \mathcal{J}(\varphi_{j})=(\varphi_{j2}(1),\varphi_{j3}(1)),
\]
respectively.
Then, we can easily check the commutativity of (\ref{EQ:diagramA}).
Let $\mathcal{P} = (\mathcal{I}_2\times\mathcal{I}_3)^{-1} \circ (\overline{\mathrm{Tr}}\times\overline{\mathrm{Tr}})^{-1} \circ \mathrm{proj}.$
Then, $\mathcal{P}$ is a $2$-$1$ surjective map satisfying $\mathcal{P} (2,b,c) = \mathcal{P} (-2,b,c)$.
As a result of the commutativity, if $\mathcal{P} (a,b,c) = (m_1,m_2)$, then we may take
\begin{equation}\label{EQ:alpha12}
\left\{
\begin{split}
\alpha_1^{(a,b,c)}&=
(
[P_{p}^{(1)} \boldsymbol{\xi}(p;\ell_1)],\,\,
[P_{k_2}^{(m_2)} \boldsymbol{\xi}(k_2;\ell_2)], \,\,
[P_{k_3}^{(m_3)} \boldsymbol{\xi}(k_3;\ell_3)]
),\\
\alpha_2^{(a,b,c)}&=
(
[P_{p}^{(\Delta)} \boldsymbol{\xi}(p;\ell_1)],\,\,
[P_{k_2}^{(m_2)} \boldsymbol{\xi}(k_2;\ell_2)], \,\,
[P_{k_3}^{(m_3)} \boldsymbol{\xi}(k_3;\ell_3)]
).
\end{split}
\right.
\end{equation}
Hence, applying (\ref{EQ:alpha12}) to $\alpha_j^{(a,b,c)}$ in (\ref{EQ:imageofalphaKDW}) gives (\ref{EQ:withp}).\qed

\appendix
\renewcommand{\thesection}{Appendix \Alph{section}}
\renewcommand{\thesubsection}{\Alph{section}.\arabic{subsection}}
\renewcommand{\theTheorem}{\Alph{section}.\arabic{Theorem}}
\section{Proofs of the lemmas in Section \ref{TRG}}\label{SEC:proofBrieskorn}
In this appendix, we provide the proofs of the lemmas and the proposition in Section \ref{TRG}.
Throughout this appendix, we adopt  the notation in Section \ref{SEC:Brieskorn}.
\subsection{Proof of Lemma \ref{aab}}
Suppose that $X \in G$ and $x \in \SL_2(\F_p)$ is a representative of $X$.
Since $-I_2 \in \SL_2(\F_p)$ is the unique matrix of order $2$,
we can verify that $X^n = I_2$ in $G$ if and only if $x^{2n} = I_2$ in $\SL_2 (\F_p)$.
Furthermore, by (\ref{OrderCheb}), we can deduce that $x^{2n}=I_2$ if and only if
\begin{equation}\label{CHECKORDER}
S_{2n}(\mathrm{Tr} x) = S_{2n-1}(\mathrm{Tr} x) +1 = 0.
\end{equation}
Since $S_{2n}(-t) = -S_{2n}(t)$ and $S_{2n-1}(-t) = S_{2n-1}(t)$, the lemma follows.\qed

\subsection{Proof of Lemma \ref{ansPSL}}
First, let us prove (i).
For $n\in\N$, one can show that $S_{2n}(0)=0$ and $S_{2n-1}(0)=(-1)^{n+1}$.
Hence, $0 \notin \mathrm{Sol}_k$ if $k$ is odd.
Furthermore, $\mathrm{Sol}_k = \overline{\mathrm{Tr}} H_k$ follows from the properties of $H_k$.

Next, let us demonstrate (ii).
From the proof of Lemma \ref{lenswithPSL}, recall $s, s'\in G$ such that $s\sigma_{T}(x)s^{-1} = \sigma_{T}(x^{-1})$ and $s'\sigma_{B}(x + \sqrt{\Delta} y)s'^{-1} = \sigma_{B}(x - \sqrt{\Delta} y)$.
Therefore, we now show the following two claims:

\noindent 
{\bf Claim 1}:
$x\in\{ y,y^{-1},-y,-y^{-1} \}$ if and only if $\overline{\mathrm{Tr}} \sigma_{T}(x) = \overline{\mathrm{Tr}} \sigma_{T}(y)$.

\noindent
{\it Proof of Claim 1}: The ``if'' part is obvious. Conversely, suppose $\overline{\mathrm{Tr}} \sigma_{T}(x) = \overline{\mathrm{Tr}} \sigma_{T}(y)$. Then, we have either $(xy -1)(x-y)=0$ or $(xy+1)(x+y) =0$. Hence, $x\in\{ y,y^{-1},-y,-y^{-1} \}$.

\noindent 
{\bf Claim 2}:
$x + \sqrt{\Delta} y \in
\{ \tilde{x} + \sqrt{\Delta} \tilde{y}, \tilde{x} - \sqrt{\Delta} \tilde{y}, -\tilde{x} + \sqrt{\Delta} \tilde{y}, -\tilde{x} - \sqrt{\Delta} \tilde{y} \}$ if and only if $\overline{\mathrm{Tr}} \sigma_{B}(x + \sqrt{\Delta} y) = \overline{\mathrm{Tr}} \sigma_{B}(\tilde{x} + \sqrt{\Delta} \tilde{y})$.

\noindent
{\it Proof of Claim 2}: The ``if'' part is obvious. 
Conversely, suppose that $\overline{\mathrm{Tr}} \sigma_{B}(x + \sqrt{\Delta} y) = \overline{\mathrm{Tr}} \sigma_{B}(\tilde{x} + \sqrt{\Delta} \tilde{y})$, i.e., $x\in \{ \pm\tilde{x} \}$.
Then, $y$, $-y$, $\tilde{y}$, and $-\tilde{y}$ are solutions of the quadratic equation $t^2 - \Delta^{-1} (1 - x^2) = 0$.
Therefore, $y\in \{ \pm\tilde{y} \}$\qed

\subsection{Proof of Lemma \ref{LEM:surjective}}
For any $(a,b,c) \in \widetilde{\mathcal{A}} \subset (\F_p)^3$, it is known from \cite[Theorem 1]{Mac} that there exist $X,Y,$ and $Z \in \SL_2(\F_p)$ such that $\mathrm{Tr} X = a,\mathrm{Tr} Y=b,\mathrm{Tr} Z=c$, and $XYZ=I_2$.
We can thus define a non-trivial homomorphism $\varphi : \Gamma \to G$ by $\varphi(x_1) = X, \varphi(x_2) = Y,$ and $\varphi(x_3) = Z$.
From the definition of $\varphi$ and Lemma \ref{aab}, it follows that $\mathcal{F}(\varphi) = (a,b,c) \in \mathcal{A}$.
This implies that $\mathcal{F} : \Hom (\Gamma, G) \setminus \{ e_G \} \to \mathcal{A}$ is surjective.
\qed

\subsection{Proof of Lemma \ref{numc}}
For $\varphi \in \mathcal{F}^{-1}(a,b,c)$, let us consider $X,Y \in \SL_2(\F_p)$ as representatives of $\varphi(x_1),\varphi(x_2) \in G$ such that $\mathrm{Tr}\, X = a$ and $\mathrm{Tr} \, Y=b$, respectively.
Set
$
X
$
as
$
\begin{pmatrix}
x_{11} & x_{12} \\
x_{21} & x_{22} 
\end{pmatrix}
$
and
$
Y
$
as
$
\begin{pmatrix}
y_{11} & y_{12} \\
y_{21} & y_{22} 
\end{pmatrix}.
$
Since $\mathcal{F}(\varphi)=(a,b,c)$, it is required that $\mathrm{Tr} (Y^{-1} X^{-1})=c$, leading to
\begin{equation}\label{aac}
\left\{
\begin{array}{l}
x_{11}x_{22} - x_{12}x_{21} = 1, \\
y_{11}y_{22} - y_{12}y_{21} = 1,
\end{array}
\right.
\quad
\text{ and }
\left\{
\begin{array}{l}
x_{11} + x_{22} = a, \\
y_{11} + y_{22} = b, \\
x_{11} y_{11} + x_{12} y_{21} + x_{21} y_{12} + x_{22} y_{22} = c.
\end{array}
\right.
\end{equation}
Applying an appropriate conjugation to $\varphi$, we may assume $X$ to be one of the representatives in Tables \ref{TAB:rep1} and \ref{TAB:rep2}.
Let us consider the three cases: (I) $k_1|p-1$,\,\,(II) $k_1=p$, and (III) $k_1|p+1$.

\noindent
(I)
Choose $\lambda\in\mu_{p-1} \setminus \{ 1\}$ such that $a = \lambda + \lambda^{-1}$, and set
$X = 
\begin{pmatrix}
\lambda & 0 \\
0 & \lambda^{-1} 
\end{pmatrix}.$
Consequently, we can simplify (\ref{aac}) to
\[
\left\{
\begin{array}{l}
y_{11} = (\lambda-\lambda^{-1})^{-1} (-\lambda^{-1} b + c), \\
y_{22} = (\lambda-\lambda^{-1})^{-1} (\lambda b - c), \\
y_{12} y_{21} = y_{11} y_{22} - 1.
\end{array}
\right.
\]
Now, as $k_1$ is an odd prime, the centralizer $C_G ( f(x_1) )$ equals $T$.
Then
\[
\sigma_{T}(z) f(x_2) \sigma_{T}(z)^{-1} = 
\begin{pmatrix}
y_{11} & z^2 y_{12} \\
z^{-2}y_{21} & y_{22} 
\end{pmatrix}
\in \PSL_2(\F_p).
\]
Hence, a complete set of representatives of $\mathcal{F}^{-1} (a,b,c)$ is $\{\varphi_1,\varphi_2\}$ such that
\[
\begin{split}
\varphi_1(x_1)=\varphi_2(x_1)&=
\begin{pmatrix}
\lambda & 0 \\
0 & \lambda^{-1} 
\end{pmatrix},\\
\varphi_1(x_2)=
\begin{pmatrix}
y_{11} & 1 \\
y_{11}y_{22} - 1 & y_{22} 
\end{pmatrix}&, \qquad
\varphi_2(x_2)=
\begin{pmatrix}
y_{11} & \Delta \\
\Delta^{-1} ( y_{11}y_{22} - 1 ) & y_{22} 
\end{pmatrix},
\end{split}
\]
where $a=\lambda+\lambda^{-1}, \,\,\,y_{11} = (\lambda-\lambda^{-1})^{-1} (-\lambda^{-1} b + c)$, and $y_{22} = (\lambda-\lambda^{-1})^{-1} (\lambda b - c)$.

\noindent 
(II)
We may suppose that either
$X=\begin{pmatrix}
1 & 1 \\
0 & 1 
\end{pmatrix}$
or
$X=\begin{pmatrix}
1 & \Delta \\
0 & 1 
\end{pmatrix}$.
Let $\lambda \in \{ 1, \Delta\}$.
One can show that (\ref{aac}) is simplified to $y_{21} = \lambda^{-1} (c-b)$.
Remark that $k_2 \neq k_3$ leads to $b \neq c$.
Now, the centralizer $C_G (f(x_1))$ equals $U$.
For $z \in \F_p$, we have 
\[
\sigma_{U}(z) Y \sigma_{U}(z)^{-1} = 
\begin{pmatrix}
y_{11} - y_{21}z & y_{12} - y_{21} z^2 + y_{11}z - y_{22} z \\
y_{21} & y_{22}+y_{21}z 
\end{pmatrix}.
\]
Therefore, taking $z=y_{21}^{-1} y_{11}$, we see that each homomorphism $\Gamma \to G$ is conjugate to $f$ in the case where $(x,y_{11})$ is either $(1,0)$ or $(\Delta,0)$.
Hence, a complete set of representatives for $\mathcal{F}^{-1} (a,b,c)$ consists of $\{\varphi_1,\varphi_2\}$ such that
\[
\begin{split}
\varphi_1(x_1) = 
\begin{pmatrix}
1 & 1 \\
0 & 1 
\end{pmatrix},& \qquad
\varphi_1(x_2) =
\begin{pmatrix}
0 & -(c-b)^{-1} \\
(c-b) & b 
\end{pmatrix}, \\
\varphi_2(x_1) =
\begin{pmatrix}
1 & \Delta \\
0 & 1 
\end{pmatrix},& \qquad
\varphi_2(x_2) = 
\begin{pmatrix}
0 & -\Delta(c-b)^{-1} \\
\Delta^{-1}(c-b) & b 
\end{pmatrix}.
\end{split}
\]

\noindent
(III) The proof provided in (I) still holds if we replace $\PSL_2(\F_p)$ with $\PSL_2(\F_{p^2})$.
In addition, if $k_1$ divides $p+1$, then $X$ is diagonalizable in $\PSL_2(\F_{p^2})$.
Hence, we can prove the case (III) in a similar fashion to (I).\qed

\subsection{Proof of Lemma \ref{numd}}
Assume $(a,b,c) \neq (a',b',c') \in \mathcal{A}$.
Let $\varphi \in \mathcal{F}^{-1}(a,b,c)$ and $\varphi' \in \mathcal{F}^{-1}(a',b',c')$.
If $f$ and $g$ are conjugate, since $\overline{\mathrm{Tr}}$ is invariant under conjugation, it must be $\{ a, -a\}=\{ a', -a'\}$, $\{ b, -b\}=\{ b', -b'\}$, and $\{ c, -c\}=\{ c', -c'\}$.
This condition is equivalent to $(a',b',c')=(a,-b,c) \in \mathcal{A}$.

Therefore, it is sufficient to consider only the case where $\varphi \in \mathcal{F}^{-1}(a,b,c)$ and $\varphi' \in \mathcal{F}^{-1}(a,-b,c)$ are conjugate.
In addition, by Lemma \ref{numc}, we may assume that both $\varphi(x_1)$ and $\varphi'(x_1)\in G$ are some representatives in Tables \ref{TAB:rep1} and \ref{TAB:rep2}.
Let $X \in \SL_2 (\F_p)$ be a representative of $\varphi(x_1) = \varphi'(x_1) \in G$ such that $\mathrm{Tr}\, X = a$.
Moreover, let $Y,Y'$ be representatives of $\varphi(x_2),\varphi'(x_2) \in \PSL_2(\F_p)$, respectively, such that $\mathrm{Tr}\, Y = b$ and $\mathrm{Tr} \, Y' = -b$.
Expand
$
Y
$
as
$
\begin{pmatrix}
y_{11} & y_{12} \\
y_{21} & y_{22} 
\end{pmatrix}
$
and
$
Y'
$
as
$
\begin{pmatrix}
y'_{11} & y'_{12} \\
y'_{21} & y'_{22} 
\end{pmatrix}.
$
Let us consider the three cases: (I) $k_1|p-1$,\,\,(II) $k_1=p$, and (III) $k_1|p+1$.

\noindent
(I)
With a choice of $\lambda \in \mu_{p-1} \setminus \{1 \}$, we may assume $X \in \sigma_{T}(\lambda)$.
Since $C_G(\sigma_{T}(\lambda)) = T$, if $\varphi$ and $\varphi'$ are conjugate, then there are some $z \in \mu_{p-1}$ such that
\[
Y = \sigma_{T}(z) Y' \sigma_{T}(z)^{-1} \in\PSL_2 (\F_p).
\]
Now, considering the $(1,1)$-entries of $Y$ and $\sigma_{T}(z) Y' \sigma_{T}(z)^{-1}$, we can see that $y_{11} \in \{ y'_{11}, -y'_{11} \}$.
Meanwhile, from the proof of Lemma \ref{numc} (I), we already know 
\[
\left\{
\begin{split}
y'_{11}-y_{11} &= 2 \lambda^{-1} (\lambda-\lambda^{-1})^{-1} b &\neq 0,\\
y'_{11}+y_{11} &= 2 (\lambda-\lambda^{-1})^{-1} c &\neq 0.
\end{split}
\right.
\]
This contradicts $y_{11} \in \{ y'_{11}, -y'_{11} \}$.

\noindent
(II)
Let us take $\lambda\in\{ 1, \Delta\}$ and assume that $X \in \sigma_U(\lambda)$.
Since $C_G(\sigma_U(\lambda)) = U$,
in a similar fashion to (I), we have $y_{21} \in \{ y'_{21},-y'_{21}\} $.
However, from the proof of Lemma \ref{numc}(II), we have
\[
\left\{
\begin{split}
y'_{21}-y_{21} &= 2 \lambda^{-1} b &\neq 0,\\
y'_{21}+y_{21} &= 2 \lambda^{-1} c &\neq 0,
\end{split}
\right.
\]
leading to a contradiction to $y_{21} \in \{ y'_{21},-y'_{21}\} $.

\noindent
(III) If $k_1$ divides $p+1$, the proof follows in a similar fashion to (I).\qed

\subsection{Proof of Proposition \ref{PROP:order}}
The first equation in (\ref{EQ:PROP:order}) holds due to
\[
\Hom (\Gamma, G) = \{ 1_G \}\sqcup\bigsqcup_{(a,b,c)\in \mathcal{A}} \mathcal{F}^{-1} (a,b,c)
\]
and Lemmas \ref{numc} and \ref{numd}.
Let us now prove the second equation in (\ref{EQ:PROP:order}).
For a non-trivial map $\varphi: \Gamma \to G$, it is sufficient to show that $g \varphi g^{-1} \neq \varphi$ for any $g \in G \setminus \{ I_2 \}$.
If $\varphi$ is non-trivial, $\varphi$ is surjective by Example \ref{sur}.
Therefore, if $g \varphi(x) g^{-1} = \varphi(x)$ holds for any $x\in \Gamma$, then $g$ must be contained in the center of $G$, which is trivial.
Hence, $g=I_2$.\qed

\subsection*{Acknowledgment}
The author sincerely expresses his gratitude to Takefumi Nosaka for his constructive suggestions and continuous support.
The author is also grateful to the anonymous referee for their careful reading and helpful comments.

This work was supported by the Research Institute for Mathematical Sciences, an International Joint Usage/Research Center located in Kyoto University.

\footnotesize{Department of Mathematics, Tokyo Institute of Technology, 2-12-1 Ookayama, Meguro-ku, Tokyo 152-8551, JAPAN \\ \texttt{e-mail:yanagida.k.ab@m.titech.ac.jp}}

\end{document}